\documentclass{amsart}

\usepackage{amsfonts}
\usepackage{amsthm}
\usepackage{amstext}
\usepackage{amsmath}
\usepackage{amscd}
\usepackage{amssymb}
\usepackage[mathscr]{eucal}
\usepackage{graphicx}
\usepackage{graphics}
\usepackage{epsf}
\usepackage{array}
\usepackage{url}

 \makeatletter
 \def\@seccntformat#1{\csname the#1\endcsname.\quad}
 \makeatother

\theoremstyle{plain}
\newtheorem{theorem}{Theorem}
\newtheorem{proposition}[theorem]{Proposition}
\newtheorem{corollary}[theorem]{Corollary}
\newtheorem{lemma}[theorem]{Lemma}

\newtheorem*{problem}{Problem}

\newcounter{mainresult}

\newtheorem{propositionA}[mainresult]{Proposition}
\newtheorem{theoremA}[mainresult]{Theorem}
\newtheorem{corollaryA}[mainresult]{Corollary}

\theoremstyle{definition}
\newtheorem{definition}[theorem]{Definition}

\theoremstyle{remark}

\newcommand\mL{L\kern-0.08cm\char39}

\newcommand{\mr}[1]{MR#1}

\newcommand{\Lip}{\operatorname{Lip}}

\newcommand{\length}[1]{\HHh^1(#1)}
\newcommand{\lengthd}[2]{\HHh^1_{#1}(#2)}
\newcommand{\IT}{I^{\mathcal{T}}}
\newcommand{\IED}{I^{\mathcal{ED}}}
\newcommand{\tof}[1]{\stackrel{#1}{\longrightarrow}}
\newcommand{\tofsmall}[1]{\stackrel{#1}{\rightarrow}}

\newcommand{\id}{{\rm id}}
\newcommand{\Endpoints}{E}
\newcommand{\End}{E}
\newcommand{\Branchpoints}{B}
\newcommand{\Branch}{B}
\newcommand{\Ordinary}{O}

\newcommand{\NNN}{\mathbb N}
\newcommand{\ZZZ}{\mathbb Z}
\newcommand{\RRR}{\mathbb R}
\newcommand{\SSS}{\mathbb S}
\newcommand{\AAA}{\mathbb{A}}

\newcommand{\CCC}{\mathbb{C}}

\newcommand{\AAa}{\mathcal{A}}
\newcommand{\BBb}{\mathcal{B}}
\newcommand{\CCc}{\mathcal{C}}
\newcommand{\DDd}{\mathcal{D}}

\newcommand{\PPp}{\mathcal{P}}
\newcommand{\RRr}{\mathcal{R}}

\newcommand{\TTt}{\mathcal{T}}
\newcommand{\TTx}{\mathcal{S}}

\newcommand{\HHh}{\mathcal{H}}

\newcommand{\oo}{o}
\newcommand{\Poo}{\PPp_{\operatorname{small}}}

\newcommand{\comment}[1]{}

\newcommand{\closure}[1]{\overline{#1}}

\newcommand{\interiorRel}[2]{\operatorname{int}_{#1}(#2)}

\newcommand{\COnSubsets}[1]{\tilde{C}(X)}

\newcommand{\eps}{\varepsilon}
\newcommand{\lev}[1]{\operatorname{lev}(#1)}

\numberwithin{equation}{section}

\newcommand{\abs}[1]{\lvert#1\rvert}
\newcommand{\card}[1]{\##1}

\newcommand{\toPar}[1]{\stackrel{#1}{\rightarrow}}

\begin{document}


\title[Topological entropy of transitive dendrite maps]
 {Topological entropy of transitive dendrite maps}

\author[V. \v Spitalsk\'y]{Vladim\'\i r \v Spitalsk\'y}
\address{Department of Mathematics, Faculty of Natural Sciences,
          Matej Bel University, Tajovsk\'eho 40, 974 01 Bansk\'a Bystrica,
          Slovakia}
\email{vladimir.spitalsky@umb.sk}
\thanks{The author was supported by the Slovak Research and Development Agency
under the contract No.~APVV-0134-10
and by the Slovak Grant Agency under the grant
VEGA~1/0978/11.}

\subjclass[2010]{Primary 37B05, 37B20, 37B40; Secondary 54H20}

\keywords{Topological entropy, transitive system, exact system, dendrite.}

\begin{abstract}
We show that
every dendrite $X$ satisfying the condition that
no subtree of $X$ contains all free arcs
admits a transitive, even exactly Devaney chaotic map
with arbitrarily small entropy.
This gives a partial
answer to a question of Baldwin from 2001.
\end{abstract}

\maketitle

\thispagestyle{empty}


\section{Introduction}\label{S:intro}
Let $X$ be a compact metric space. By a \emph{dynamical system} on
$X$ we mean a pair $(X,f)$, where $f:X\to X$ is a continuous
(not necessarily injective) map.
The topological entropy of $f$, defined in \cite{AKM}, will be
denoted by $h(f)$.
This is a fundamental quantitative
characteristic of a dynamical system. On the other hand, a dynamical
system can possess some qualitative dynamical properties, such as
transitivity, weak and strong mixing and exactness (see
Section~\ref{S:preliminaries} for definitions). It is natural to
study the connection between these qualitative properties and
entropy. The first result in this direction, due to Blokh
\cite{Blo82}, says that the topological entropy of any transitive
dynamical system on the interval is at least $(1/2) \log2$; moreover, he
showed that this bound is the best possible. Hence we have
$\IT([0,1])=(1/2)\log2$, where $\IT(X)$ is defined by
$$
 \IT(X)=\inf\{
  h(f):\ f:X\to X \text{ is transitive}
 \}.
$$
Here, the infimum is \emph{attainable} since there is a transitive interval map
with the entropy $(1/2)\log2$.
In
1987, Blokh \cite{Blo87} showed that transitivity implies
positiveness of entropy for maps on any graph which is not a circle.
Later Alsed\`a,
Kolyada, Llibre and Snoha \cite{AKLS} showed that for an $n$-star $S$
it holds that $\IT(S)=(1/n)\log2$.
For arbitrary tree $T$, Alsed\`a, Baldwin, Llibre and
Misiurewicz \cite{ABLM} showed that $\IT(T)\ge (1/n)\log2$, where $n$
is the number of end points of the tree. For some classes of trees
this bound is the best possible (see \cite{Ye} and \cite{Bal01})
but in general it is still an open problem to determine the value of
$\IT(T)$ for arbitrary tree $T$. By Alsed\`a, del R\'io and Rodr\'iguez
\cite{ARR}, for any graph $X$
which is not a tree it holds that $\IT(X)=0$; they even showed that
$\IED(X)=0$, where
$$
 \IED(X)=\inf\{
  h(f):\ f:X\to X \text{ is exactly Devaney chaotic}
 \}.
$$
By \cite[Lemma~8.5]{HKO11}, every tree admitting a transitive $P$-linear Markov
map admits, for every $\eps>0$, also an exact $P'$-linear Markov map $f'$ with
$h(f')<h(f)+\eps$. Thus, for trees $T$, the upper bounds for $\IT(T)$ from \cite{Ye}
are true also for $\IED(T)$; in particular, $\IED(T)=(1/n)\log 2$
for every $T\in \TTt_0$ (for the definitions of Ye's classes $\TTt_i$ of trees, see
e.g.~Subsection~\ref{SS:trees}).

A natural class of compact metric spaces containing all trees is the
class of dendrites. Recall that a \emph{dendrite} is a locally connected
continuum containing no
simple closed curve. Dendrites possess many properties of trees:
they have the fixed point property, they are absolute retracts, they
are embeddable in the plane, etc. On the other hand, dendrites can be much more
complicated than trees; for example there is an uncountable system of
pairwise non-homeomorphic dendrites. Recently, the study of dynamics
on dendrites attracted many authors; see for instance
\cite{Il98, Kato98, AEO07, Baldwin07, Ma09, SWZ10, Na12} and other references therein.
If $X$ is a dendrite we denote by
$\Endpoints(X)$ the set of all \emph{end points} of $X$ (that is,
points $x\in X$
such that $X\setminus\{x\}$ is connected) and
by $\Branchpoints(X)$ the set of all \emph{branch points} of $X$
(that is, points $x\in X$
such that $X\setminus\{x\}$ has at least three components).
Notice that, for a non-degenerate dendrite $X$, $\Branchpoints(X)$ is always countable
and it is empty if only if
$X$ is an arc. Further, the set of end points $\Endpoints(X)$
is always non-empty and it is finite if and only if $X$ is a tree.

In this paper we study infima $\IT(X),\IED(X)$ for
non-degenerate dendrites which are not trees.
Every such dendrite $X$ has infinitely many end points and so,
for every $n$, it contains an $n$-star $S_n$ or an $n$-comb $C_n$
(see Corollary~\ref{C:treesInDendrite}). Since
$\IT(S_n)=(1/n)\log 2$ by \cite{AKLS} and $\IT(C_n)=(1/n)\log 2$ for $n=2^r$ by \cite{Ye},
the dendrite $X$ contains subtrees admitting transitive maps with arbitrarily small
entropy. So it is natural to expect that the answer to the following Baldwin's question
is affirmative.

\begin{problem}[\cite{Bal01}]
Is it true that $\IT(X)=0$ for every non-degenerate dendrite $X$ which is not a tree?
\end{problem}

The purpose of the present paper is to give a partial affirmative answer to this question.
Roughly speaking, we show that $\IT(X)=\IED(X)=0$ provided $X$ contains ``sufficiently
many'' free arcs. Recall that an \emph{arc} is a homeomorphic image of $[0,1]$ and a
\emph{free arc} is an arc $A$ such that every non-end point of $A$ is an interior point of $A$
in the topology of $X$.

\begin{theoremA}\label{T:main}
 If $X$ is a non-degenerate dendrite such that no subtree of it contains all
 free arcs of $X$
 then $X$ admits a transitive, even exactly Devaney chaotic
 map with arbitrarily small entropy. Hence
 $$
  \IT(X)=\IED(X)=0,
 $$
 where the infima are not attainable.
\end{theoremA}

The class of dendrites satisfying the condition from Theorem~\ref{T:main}
has not been studied in the literature, but it appears to be natural for our
construction. Note that in Lemma~\ref{L:subtreeFreeArcs} we will show that
the condition is satisfied e.g.~for dendrites which are not trees and
which have nowhere dense branch points.

Combining Theorem~\ref{T:main} with the main result of \cite{Sp12b}
we obtain the following corollary. Recall that a continuum $X$ is
\emph{completely regular} provided the only nowhere dense subcontinua of $X$
are degenerate. Since graphs are completely regular, the corollary
generalizes \cite[Theorem~4.1]{ARR}.

\begin{corollaryA}\label{C:totallyRegular}
 If $X$ is a non-degenerate completely regular continuum which is not a tree then
 $\IT(X)=\IED(X)=0$.
\end{corollaryA}

As a by-product, by combining our results with the constructions from
\cite{Ye} we obtain the following upper estimate of $\IED(X)$
for dendrites $X$ based on the number of end points of special subtrees of $X$.
Here we say that an edge of a tree~$T$ is \emph{terminal} if it contains an end point of $T$.
For the definition of Ye's class $\TTt_0$, see the end of Subsection~\ref{SS:trees}.

\begin{propositionA}\label{P:upperBoundFreeArc}
 If a dendrite $X$
 contains a tree $T$ from Ye's class $\TTt_0$
 with $n$ end points and
 with all terminal edges having non-empty
 interiors in $X$, then
 $$
  \IT(X)\le\IED(X)\le \frac{\log 2}{n}\,.
 $$
 In particular, if $X$ contains a free arc then
 $\IED(X)\le (1/2)\log2$.
\end{propositionA}



The paper is organized as follows. In the next section we summarize
all the definitions and facts which will be needed later and
we also prove some elementary facts about dendrites.
In Section~\ref{S:markov} we introduce the so-called $(P,S)$-linear Markov
maps on trees and we show how they can be used to define exactly Devaney chaotic
tree maps with small entropy.
In Section~\ref{S:exactOnDendrites} we prove the main results of the paper.
Finally, in two appendices we give simple
constructions of zero-entropy $(P,S)$-linear Markov maps on stars and combs,
and we show, for trees $T$,
an upper bound for $\IED(T)$ which depends only on the number of end points of $T$.


\section{Preliminaries}\label{S:preliminaries}

In this section we
recall basic notations and definitions which are used in the paper.
The cardinality of a set $A$ is denoted by $\card{A}$.
By $\ZZZ$ ($\NNN$, $\NNN_0$) we denote the set of all
(positive, non-negative) integers, by $\RRR$ we denote the set of
all reals and
by $I$ we denote the unit interval $[0,1]$.
For an interval $J$, $\abs{J}$ is the length (that is, the Lebesgue measure)
of it.
The symbol $\sqcup$ denotes a disjoint union.

A \emph{space} is any topological space; it is \emph{degenerate} if
it consists of one point, otherwise it is \emph{non-degenerate}.
A \emph{Cantor space} is a totally disconnected
compact metric space without isolated points.
A set is said to be \emph{regular closed} if it is the
closure of its interior or, equivalently,
if it is the closure of an open set. 

A \emph{dynamical system} is a pair $(X,f)$, where $X$ is
a compact metric space and $f:X\to X$ is a
continuous selfmap of $X$. For non-negative integers $n$
we define the \emph{$n$-th iterate} $f^n$
of $f$ recursively as follows: $f^0$ is the identity map on $X$ and
$f^{n}=f\circ f^{n-1}$ for any $n\in\NNN$.
A subset $M$ of $X$ is called \emph{$f$-invariant}
(\emph{strongly $f$-invariant}) provided $f(M)\subseteq M$ ($f(M)=M$).
A \emph{subsystem} of $(X,f)$ is any dynamical system $(M,f|_M)$ where $M$ is a
non-empty closed $f$-invariant subset of $X$.
If $(X,f)$ is a dynamical system,  a point $x$ of $X$
is called \emph{fixed} if $f(x)=x$ and
\emph{periodic} if $f^p(x)=x$ for some $p\in\NNN$.

The topological entropy of a system $(X,f)$ will be denoted by $h(f)$.
A system $(X,f)$ is called
\emph{transitive} if for every non-empty open sets $U,V$ there is
$n\in\NNN$ with $f^n(U)\cap V\ne\emptyset$.
It is \emph{exact} if for every non-empty open set $U$ there is
	 $n\in\NNN$ with $f^n(U)=X$.
A system $(X,f)$ is \emph{(exactly) Devaney chaotic} if $X$ is infinite
and $f$ is transitive (exact) and has dense periodic points.

\subsection{Perron-Frobenius theory}\label{SS:PerronFrobenius}
Let $M=(m_{ij})_{ij=1}^n$
be an $n\times n$ matrix with real entries.
A sequence $\pi=(p_j)_{j=0}^k$ ($k\ge 1$) of elements
of $\{1,\dots,n\}$ is called a \emph{path} if its \emph{width}
$w(\pi)=\prod_{j=1}^k m_{p_{j-1} p_j}$ is non-zero; the \emph{length} of $\pi$
is $\abs{\pi}=k$. A path $(p_j)_{j=0}^k$ is called a \emph{loop} if $p_0=p_k$.
The matrix $M$ is called \emph{irreducible} if for every $i,j$ there is a path
from $i$ to $j$. If $M$ is irreducible and the greatest common divisor of lengths
of all loops is equal to $1$, $M$ is said to be \emph{primitive}.

By Perron-Frobenius theorem (see e.g.~\cite[Theorems~1.1 and 1.5]{Seneta}),
every non-negative irreducible matrix $M$
has an eigenvalue $\lambda_M>0$ of multiplicity one
such that $\lambda_M\ge \abs{\lambda}$
($\lambda_M> \abs{\lambda}$ if $M$ is primitive)
for every eigenvalue $\lambda\ne\lambda_M$ of $M$; this $\lambda_M$ is called
the \emph{Perron eigenvalue} of $M$. Moreover, there are left and right
eigenvectors of $M$ associated with $\lambda_M$ which are strictly positive.
If $M$ is a general (that is, not necessarily irreducible) non-negative matrix,
the spectral radius $\lambda_M$ of it
(which is equal to the maximal non-negative eigenvalue of $M$)
will be also called the \emph{Perron eigenvalue} of $M$.
Notice that if $M,N$ are non-negative matrices
and $M\le N$ (that is, $m_{ij}\le n_{ij}$ for every $i,j$), then
$\lambda_M\le\lambda_N$.

In the special case when $M$ is an irreducible $01$-matrix, we have $\lambda_M\ge 1$;
moreover, $\lambda_M=1$ if and only if $M$ is a \emph{permutation matrix}
(that is, it has exactly
one entry $1$ in each row and each column and $0$'s elsewhere), see e.g.~\cite[p.~8]{Seneta}.
In particular, if $M$ is a primitive $01$-matrix, then $\lambda_M>1$.


The so-called \emph{rome method} for calculating the characteristic polynomial of
square matrices was proposed in \cite[p.~20--21]{BGMY}.
Let $M=(m_{ij})_{ij=1}^n$
be an $n\times n$ matrix.
A non-empty subset $\RRr\subseteq \{1,\dots,n\}$ is called a \emph{rome}
if for every loop $(p_j)_{j=0}^k$ there is $j$ with $p_j\in\RRr$.
Let $\RRr=\{r_1,\dots,r_k\}$ ($r_i\ne r_j$ for $i\ne j$) be a rome.
A path $(p_j)_{j=0}^k$ is said to be \emph{simple} (or \emph{$\RRr$-simple})
if $p_0,p_k\in\RRr$ and $p_j\not\in\RRr$ for $1\le j<k$.
Define a matrix function $R_{\RRr}$ by
$$
 R_{\RRr}=(r_{ij})_{ij=1}^k,
 \qquad
 r_{ij}(x) = \sum_\pi w(\pi) x^{-\abs{\pi}}
 \quad (x\in\RRR),
$$
where the summation is over all simple paths $\pi$ from $r_i$ to $r_j$.
Let $E$ denote the unit matrix (of an appropriate size). By \cite[Theorem~1.7]{BGMY},
the characteristic polynomial $\chi_M(x)=\det(M-xE)$ of $M$ satisfies
\begin{equation}\label{EQ:romeMethod}
 \chi_M(x) = (-1)^{n-k} x^n \cdot \det(R_{\RRr}(x)-E)
 \qquad
 \text{for every rome }\RRr \text{ and }  x\in\RRR.
\end{equation}
Thus, if $M$ is non-negative with positive spectral radius $\lambda_M$, then
$\lambda_M$ is equal to
the maximal positive root of $\det(R_{\RRr}(x)-E)$.

\subsection{Continua}\label{SS:continua}
A \emph{continuum} is
any connected compactum.
An \emph{arc} is a homeomorphic image of the unit interval $I$; the images of
$0,1$ are called the \emph{end points} of the arc.
An arc $A$ with end points $\{a,b\}$ in a space $X$ is \emph{free} if
$A\setminus\{a,b\}$ is open in $X$. A free arc $A$ in a connected space $X$ is said to
\emph{disconnect} $X$ if $X\setminus \{x\}$ is not connected for every
non-end point $x$ of $A$.
A \emph{circle} or a \emph{simple closed curve} is a homeomorphic image
of the unit circle $\SSS^1$.

Let $X$ be a continuum. A metric $d$ on $X$ is said to be \emph{convex}
provided for every distinct $x,y\in X$
there is $z\in X$ such that $d(x,z)=d(z,y)=d(x,y)/2$.
By \cite[Theorem~8]{BingPartSet} every locally connected
continuum admits a compatible convex metric.
If $X$ is endowed with a convex metric $d$, then for every $a\ne b$ there is an arc
$A$ with end points $a,b$, the length $\lengthd{d}{A}$ of which is
equal to $d(a,b)$; every
such arc will be called \emph{geodesic}.

\subsection{Trees}\label{SS:trees}
A \emph{graph} is a continuum which can be written as the union of finitely many
arcs, any two of which are either disjoint or intersect only
in one or both of their end points.
A \emph{tree} is a graph containing no simple closed curve.
We assume that every tree has associated with it
finitely many points of it, called \emph{vertices}, in such a way
that every end point as well as every branch point is a vertex (so we allow
also vertices of order $2$). The edges of a tree are the free arcs
both end points of which are vertices.
An edge $E$ of a tree $T$ is said to be \emph{terminal} provided
it contains an end point of $T$.

Let $n\ge 3$. Then a tree $T$ is an \emph{$n$-star}
if $T$ has exactly one branch point and this branch point has order $n$.
Further, a tree $T$ is an \emph{$n$-comb} if it has exactly $n-2$ branch points,
every branch point has the order $3$ and all the branch points belong to an arc.
We will also say that a tree $T$ is a \emph{$2$-star} or a \emph{$2$-comb}
if it is an arc.
Notice that $n$-stars and $n$-combs have exactly $n$ end points and
that every subtree of a comb (star) is a comb (star).

The following simple lemma will be used later.

\begin{lemma}\label{L:limitOfI(T_n)-2}
 Let $k\ge 3$ and $p\ge 0$ be integers.
 Let $T$ be a tree with at least $k^{p}$ branch points
 and such that every branch point of $T$
 has the order at most $k$.
 Then $T$ contains a $(p+3)$-comb.
\end{lemma}
\begin{proof}
 We prove the lemma by induction. For $p=0$ it is trivial.
 Assume now that the assertion of the lemma holds for some $p\ge 0$;
 we are going to prove that it holds also for
 $p+1$. Let $T$ be any tree with at least $k^{p+1}$ branch points and
 such that the order of every branch point is at most
 $k$. Take an end point $e$ of $X$ and let $b$ be the branch point of $X$ nearest to $e$;
 then $[b,e]$ is a free arc.
 The open set $X\setminus\{b\}$ has at most $k$ components the closures of which
 are trees with the total number of branch points at least $k^{p+1}-1$.
 Hence there is a component $U$ of
 $X\setminus\{b\}$ the closure of which
 has at least $k^{p}$ branch points. Since
 every branch point of $\closure{U}$ has the order at most $k$,
 by the induction hypothesis $\closure{U}$ contains a $(p+3)$-comb
 $S$; we may assume that $b$ is an end point of $S$.
 Let $e'$ be an end point of $X$ contained in a component of $X\setminus\{b\}$ different from both $U$ and $(b,e]$.
 Then $S\cup[b,e]\cup[b,e']$ is a $(p+4)$-comb.
\end{proof}


In the following we recall Ye's classification of trees \cite[p.~298-299]{Ye}.
Let $T_0,\dots,T_n$ be trees, let $e_1,\dots,e_n$ be the end points of $T_0$ and
$e_j'$ be a non-end point of $T_j$ ($j=1,\dots,n$). We say that a tree $T$
is an \emph{extension} of $T_1,\dots,T_n$ over $T_0$
if $T$ is homeomorphic to the free union $\bigoplus_{j=0}^n T_j$ with the pairs
$(e_j,e_j')$ ($j=1,\dots,n$) identified.

Let $r\ge 1$, $n_1,\dots,n_r\ge 2$ be integers and let $T$ be a tree.
We say that $T$ belongs to $\TTt_0^{n_1,\dots,n_r}$ if either $T$ is an $n_1$-star
(if $r=1$) or $T$ is an extension of some
$T_1,\dots,T_{n_r}\in \TTt_{0}^{n_1,\dots,n_{r-1}}$ over an $n_r$-star (if $r\ge 2$).
If $i>0$ we say that $T$ belongs to $\TTt_i^{n_1,\dots,n_r}$
if $T$ is a free union of some $T'\in\TTt_{i-1}^{n_1,\dots,n_r}$ with an arc $A$,
where a non-end point of $T'$ is identified with an end point of $A$.
Notice that a tree $T$ from $\TTt_i^{n_1,\dots,n_r}$ has $n=i+\prod_{l=1}^r n_l$ end points.

Denote by $\TTt_0$ the union of all $\TTt_0^{n_1,\dots,n_r}$ and, for $i\ge 1$,
define inductively $\TTt_i$ as the union of all $\TTt_i^{n_1,\dots,n_r}$
minus $\bigcup_{0\le j<i} \TTt_j$.
Notice that every $n$-star as well as every
$2^r$-comb ($r\ge 1$) belongs to $\TTt_0$; indeed, for $n$-stars choose $r=1, n_1=n$
and for $2^r$-combs put $n_1=\dots=n_r=2$.


\subsection{Dendrites}\label{S:treesInDendrites}
A \emph{dendrite} is a (possibly degenerate)
locally connected continuum containing no simple closed curve.
If $X$ is a dendrite, a point $x\in X$ is called an \emph{end (ordinary, branch) point}
of $X$ if $X\setminus\{x\}$ has one (two, at least three) connected components. The sets
of all end, ordinary and branch points of $X$ are denoted by $\End(X)$, $\Ordinary(X)$
and $\Branch(X)$, respectively.
A dendrite is a tree if and only if it is non-degenerate and the set of end points
is finite.

Notice that dendrites are \emph{uniquely arcwise connected}, that is,
for every $a\ne b$ there is a unique arc with end points $a,b$.
If $X$ is a dendrite and $a,b$ are different points of $X$, the arc
with end points $a$ and $b$ will be denoted by
$[a,b]$. The set $[a,b]\setminus\{a,b\}$ will be denoted by $(a,b)$;
analogously we define $(a,b]$ and $[a,b)$.

Let $X$ be a dendrite and $Y$ be a subcontinuum of $X$. The \emph{first point
retraction} $r=r_Y:X\to Y$ is the map defined by $r(y)=y$ for $y\in Y$
and $r(x)=y_x$ for $x\in X\setminus Y$, where $y_x$ is the unique point of
the boundary of the component of $X\setminus Y$ containing $x$; this map
is well defined and continuous, see e.g.~\cite[Lemma~10.25]{Nad}.

The following two lemmas are simple statements about dendrites.
We prove just the second one since the first one is trivial.

\begin{lemma}\label{L:dendriteWithFiniteBX}
 If $X$ is a non-degenerate
 dendrite with finitely many branch points then either
$X$ is a tree or
$X$ contains an $\omega$-star $S$ which is a union of free arcs.
\end{lemma}

\begin{lemma}\label{L:dendriteWithInfiniteBX}
 If $X$ is a dendrite with infinitely many branch points then
 at least one of the following holds:
 \begin{enumerate}
	\item[(a)] $X$ contains an arc containing infinitely many branch points;
	\item[(b)] $X$ contains an $\omega$-star which is a union of free arcs.
 \end{enumerate}
\end{lemma}
\begin{proof} 
 Since $B(X)$ is infinite there is a sequence of branch points $(b_n)_n$ converging to a point
 $b\not\in\{b_n:\ n\in \NNN\}$ (notice that $b$ need not be a branch point).
 Assume now that there
 is a component $U$ of the set $X\setminus\{b\}$ containing
 infinitely many $b_n$'s; without loss
 of generality we may assume that $b_n\in U$ for every $n$. Define the dendrite $Y$ to
 be the closure of $U$ (that is, $Y=U\cup\{b\}$) and take any end point
 $a$ of $Y$ different from $b$; we show that the arc $A=[a,b]$
 contains infinitely many branch points.
 Indeed, let $r=r_A:Y\to A$ be the first point retraction of $Y$ onto $A$.
 Since $r$ is  continuous, the points $b_n'=r(b_n)\in A$ converge to $r(b)=b$. Notice that
 every $b_n'$ is a branch point of $Y$, hence is a branch point of $X$.
 This is immediate if $b_n\in A$;
 otherwise $b_n\not\in A$, so $b_n'\ne b_n$ and $b_n'$
 is not an end point of $A$ (since $b_n'$ is not an end point of $Y$),
 thus $a,b,b_n$ belong to different components of $Y\setminus\{b_n'\}$.
 We also have that $b_n'\ne b$ for every $n$
 due to the fact that $b$ is an end point of $Y$. Now using the fact that $b_n'\to b$ we
 obtain that $\{b_n':\ n\in\NNN\}$ is an infinite set of branch points
 belonging to the arc $A$.

 If every component of $X\setminus\{b\}$ contains only finitely
 many branch points of $X$ then
 we have two facts.
 First, $X\setminus\{b\}$ has infinitely many components $U_0,U_1,\dots$.
 Second, every component
 $U_i$ contains a point $b_i$ such that $[b,b_i]$ is a free arc in $\closure{U_i}$,
 hence it is a free
 arc in $X$. So in this case we have that $X$ contains an $\omega$-star
 $\bigcup_{i=0}^{\infty} [b,b_i]$ with every $[b,b_i]$ being a free arc.
\end{proof}

\begin{corollary}\label{C:treesInDendrite}
 Let $X$ be a non-degenerate dendrite which is not a tree. Then for every $n$ at least one of
 the following holds:
 \begin{enumerate}
	\item[(a)] $X$ contains an $n$-comb;
	\item[(b)] $X$ contains an $n$-star which is a union of free arcs.
 \end{enumerate}
\end{corollary}

\begin{lemma}\label{L:subtreeFreeArcs}
For a non-degenerate dendrite $X$ the following are equivalent:
\begin{enumerate}
	\item[(a)] no subtree $T$ of $X$ contains all free arcs of $X$;
	\item[(b)] for every $n\ge 2$ there is a subtree $T$ of $X$ with $n$ end points,
	 the terminal edges of which have non-empty interiors in $X$;
	\item[(c)] for every $n\ge 2$ there is an $n$-star or an $n$-comb $T$ in $X$,
	 the terminal edges of which have non-empty interiors in $X$.
\end{enumerate}
Moreover, the properties (e), (f) are equivalent, they imply (d) and, if
$X$ is not a tree, (d) implies (a)--(c), where
\begin{enumerate}
	\item[(d)] $B(X)$ is nowhere dense (that is, free arcs are dense in $X$);
	\item[(e)] every arc in $X$ has non-empty interior;
	\item[(f)] $X$ is completely regular.
\end{enumerate}
\end{lemma}

Notice that, in the assertions (b) and (c),
instead of the statement ``the terminal edges of which have non-empty interiors in $X$''
we can write ``the terminal edges of which are free arcs in $X$''
(recall that in trees we allow vertices of order $2$).
Notice also that (d) is not equivalent to (e)--(f);
e.g.~any dendrite such that all branch points of it belong to an arc $A$ and are
dense in $A$,
satisfies (d) but not (e)--(f).

\begin{proof}
 For the purpose of the proof let $\TTx$ denote the system of all subtrees $T$ of $X$ such that
 every terminal edge of $T$ has non-empty interior in $X$.

 (c) implies (b) is immediate.
 Also (b) implies (a) is trivial; to show this assume, on the contrary,
 that (b) holds and there is a tree $T$ containing all free arcs of $X$.
 Fix any $n> \card E(T)$ and
 let $T_n$ be a subtree of $X$ with $n$ end points such that the terminal edges of $T_n$
 are free arcs in $X$. Then $T$ contains every end point of $T_n$ and,
 being uniquely arcwise connected, it has at least $n$ end points --- a contradiction.

 The fact that (b) implies (c) follows from Lemma~\ref{L:limitOfI(T_n)-2}.
 In fact, for $n=2,3$ the assertion is trivial,
 so let $n\ge 4$ and let $T\in \TTx$ be a tree with
 $m\ge (n-1)^{n-1}$ end points.
 If $T$ has a vertex of order at least $n$
 then it contains an $n$-star $S$; we may assume that
 $E(S)\subseteq E(T)$. Then $S\in\TTx$. Otherwise by Lemma~\ref{L:limitOfI(T_n)-2}
 we have that $T$ contains an $(n+2)$-comb $S'$;
 we may again assume that $E(S')\subseteq E(T)$ and
 so $S'\in\TTx$.

 Now we prove that (a) implies (b).
 For a subtree $T$ of $X$ let $F_T$ denotes the union of all free arcs $A\subseteq T$.
 Assume that (a) is true;
 to show (b) it suffices to prove
 that the number of end points of trees $T\in\TTx$
 is not bounded from above. Assume, on the contrary,
 that $n=\sup_{T\in\TTx} \card E(T) <\infty$. Take $T\in\TTx$ with $n=\card E(T)$;
 we may assume that $E(T)\subseteq E(X)$. Let $S$ be the minimal subtree of $T$ containing
 $F_T$; then $\card E(S)=n$ and $S\in\TTx$. Denote by $x_1,\dots,x_n$
 the end points of $S$.
 Let $(V_k)_{k\in K}$ be the (countable) system of (open) components of
 $X\setminus \Endpoints(S)$ disjoint from $S$
 and let $x_{m_k}$ be the boundary point of $V_k$ ($k\in K$).
 Let $K'\subseteq K$ be the set
 of all $k$'s such that $Y_k=\closure{V}_k$ has a free arc. Since, by (a),
 $S$ does not contain all free arcs
 of $X$, the set $K'$ is non-empty. Moreover, if
 $k\in K'$ then $x_{m_k}\in \Endpoints(S)$ and if
 $k\ne k'$ belong to $K'$
 then $x_{m_k}\ne x_{m_{k'}}$ since
 otherwise we can easily construct a tree $S'\supseteq S$ from $\TTx$
 with than $(n+1)$ end points. An analogous argument shows that for every $k\in K'$
 there is an arc $A_k$ with $x_{m_k}\in\Endpoints(A_k)$
 containing all free arcs of $Y_k$. But then
 $S' = S\cup \bigcup_{k\in K'} A_k$ is a subtree of $X$ (use that $K'$ is finite)
  containing all free arcs of $X$ --- a contradiction.

 To show that (e) implies (d) assume that $B(X)$ is not nowhere dense. Take a non-empty
 connected open subset $V\subseteq \closure{B(X)}$ and put $Y=\closure{V}$.
 Since $B(Y)=B(X)\cap V$
 and $B(X)$ is dense in $V$ we have that $\closure{B(Y)}=Y$. Hence any arc $A\subseteq Y$
 has empty interior in $Y$ and so it has empty interior in $X$.

 The fact that (f) implies (e) is immediate since every arc $A\subseteq X$ is a
 non-degenerate subcontinuum of $X$.
 The condition (e) implies (f) since if $Y$
 is a non-degenerate nowhere dense subcontinuum of $X$
 then it contains a
 nowhere dense arc $A\subseteq Y$.
 Finally, if $X$ is not a tree then (d) implies (a),
  since if branch points are nowhere dense
 then free arcs are dense in $X$;
 hence the only subcontinuum of $X$ containing
 all free arcs of $X$ is $X$ itself.
\end{proof}

\subsection{Totally regular continua}

A continuum $X$
is \emph{totally regular} if for every $x\in X$ and every countable set
$P\subseteq X$ there is a basis of neighborhoods of $x$ with finite boundary not intersecting $P$.
Every dendrite, as well as every \emph{completely regular} continuum
(that is, continuum containing no non-degenerate nowhere dense subcontinuum)
is totally regular.
Totally regular continua are locally connected one-dimensional continua.
By \cite{EH}, every totally regular continuum admits a compatible convex metric $d$ such that
$(X,d)$ has finite one-dimensional Hausdorff measure $\lengthd{d}{X}$.

\subsection{Length-expanding Lipschitz maps}
Here we prove a slight refinement of the main result of \cite{Sp12a}
valid for dendrites. First we recall the definition of length-expanding
Lipschitz maps. If $X$ is a
non-degenerate totally regular continuum
and $\CCc$ is a family of
non-degenerate subcontinua of $X$,
we say that $\CCc$
is \emph{dense} if every non-empty open set in $X$ contains a member of $\CCc$.
By $\CCc_I$ we denote the system of all non-degenerate closed subintervals of $I$;
the Euclidean metric on $I$ is denoted by $d_I$.
A map $f$ is called \emph{Lipschitz-$L$} if $\Lip(f)\le L$.

\begin{definition}[\cite{Sp12a}] Let $X=(X,d)$, $X'=(X',d')$
be non-degenerate (totally regular) continua of finite length and let
$\CCc,\CCc'$ be dense systems of
subcontinua of $X,X'$, respectively.
We say that a continuous map
$f:X\to X'$ is \emph{length-expanding} with respect to $\CCc,\CCc'$
if there exists $\varrho>1$
such that for every $C\in \CCc$ we have $f(C)\in\CCc'$ and
\begin{equation*}\label{EQ:defLengthExpanding}
       \text{if} \quad
       f(C)\ne X'
       \qquad\text{then}\quad
       \lengthd{d'}{f(C)} \ge \varrho\cdot \lengthd{d}{C}.
\end{equation*}
Moreover, if $f$ is surjective and Lipschitz-$L$ we say
that $f:(X,d,\CCc)\to (X',d',\CCc')$ is \emph{$(\varrho,L)$-length-expanding Lipschitz}.
Sometimes we briefly say that $f$ is
\emph{$(\varrho,L)$-LEL} or only \emph{LEL}.
\end{definition}
In this paper we will use the above notions also in the case when $\varrho\in (0,1]$;
to distinguish them, in such a case we will say that $f$ is \emph{length-expanding$^*$} or \emph{$(\varrho,L)$-LEL$^*$}.

Let $X$ be a non-degenerate dendrite.
By e.g.~\cite[Theorem~10.27]{Nad}, we may write $X$ in the form
\begin{equation}\label{EQ:XasUnionXi}
\begin{split}
 &X=\operatorname{closure}\left(
     \bigcup_{n\in N}  X_n
   \right),
 \qquad
 \text{where }
 X_n=A_1\cup\dots\cup A_n,
\\
&A_n=[a_n,b_n] \text{ are arcs with }
 A_n\cap X_{n-1}=\{a_n\} \text{ for } n\ge 2
\end{split}
\end{equation}
(here $N=\NNN$ provided $X\ne A_1$ and $N=\{1\}$ otherwise).
Define a linear ordering on every $A_n$ in such a way that $a_n<b_n$.
We will need the following refinement of \cite[Theorem~C]{Sp12a}.
Till the end of the paper we fix a constant $q\in (0,1/3]$.

\begin{proposition}\label{P:Sp12a-ThmC}
Let $X$ be a dendrite and $a\ne b$ be points of it.
Write $X$ in the form
(\ref{EQ:XasUnionXi}) with $a_1=a,b_1=b$
and fix convex metrics $d_n$ on $A_n$ ($n\in\NNN$).
Then there are a convex metric
 $d=d_{X,a,b}$ on $X$ and Lipschitz surjections
 $\varphi=\varphi_{X,a,b}:I\to X$, $\psi=\psi_{X,a,b}:X\to I$
 with the following properties:
\begin{enumerate}
  \item[(a)] $\lengthd{d}{X}= 1$, $d(a,b)>1-q$;
  \item[(b)] the system
    $$
     \CCc=\CCc_{X,a,b}=\{\varphi(J):\ J\text{ is a non-degenerate
     closed subinterval of } I\}
    $$
    of subcontinua of $X$ is dense;
   \item[(c)] the maps
   $$
    \varphi:(I,d_I,\CCc_I)\to (X,d,\CCc)
    \quad\text{and}\quad
    \psi:(X,d,\CCc)\to (I,d_I,\CCc_I)
   $$
   are $(\varrho,L)$-LEL$^*$, where $0<\varrho<1<L$ are constants depending only on
   $q$,
   and $\varphi(0)=a$, $\varphi(1)=b$, $\psi(a)=0$, $\psi(b)=1$;
  \item[(d)] there is $c>0$ such that $\psi(x)=c \cdot d(a,x)$
   for every $x\in X$;
  \item[(e)] for every $n$ there is $c_n>0$ such that $d(x,y)=c_n \cdot d_n(x,y)$ whenever
   $x,y\in A_n$;
  \item[(f)] $\lengthd{d}{A_{n+1}} \le q\cdot \lengthd{d}{A_n}$ for every $n$;
  \item[(g)] if $r=r_{A_1}:X\to A_1$ is the first point retraction of $X$ onto $A_1$,
   then
   $$
    r\circ \varphi(s) \le  r\circ \varphi(t)
    \qquad\text{for every }s\le t \text{ from } I.
   $$
\end{enumerate}
\end{proposition}
Notice that, if $X=[a,b]$ is an arc, then $\CCc_{X,a,b}$ is the system of all subarcs
of $X$.

\begin{proof}
If $X=[a,b]$ the statement is trivial; so assume that $N=\NNN$.
For every $n$ let $f_n:X_{n+1}\to X_n$ be the first point retraction
(that is, $f_n(x)=x$ for $x\in X_n$ and $f_n(x)=a_{n+1}$ for $x\in A_{n+1}$).
By Anderson-Choquet embedding theorem
(see e.g.~\cite[2.10]{Nad}),
 $X$ is homeomorphic to the inverse limit $\varprojlim (X_n,f_n)$.
Repeating the construction from \cite[Section~5.2]{Sp12a}
(see also the remark after \cite[Corollary~26]{Sp12a}) 
word-by-word, with $\tilde{d}_n=c_n\cdot d_n$ where $c_n$'s are sufficiently small,
we obtain $d_{X,a,b}$, $\varphi_{X,a,b}$ and $\psi_{X,a,b}$ such that
all the properties (a)--(g) are satisfied.
\end{proof}

The following can be proved in the same way as \cite[Theorem~D]{Sp12a}.
For $k\in\NNN$, let $f_k:I\to I$ denote the continuous map
which fixes $0$ and maps every $[(i-1)/k,i/k]$ linearly onto $I$.

\begin{proposition}\label{P:Sp12a-ThmD}
Keeping the notation from Proposition~\ref{P:Sp12a-ThmC},
for every $\rho>1$, every non-degenerate totally regular continua $X,X'$ and
every pairs of distinct points
$a,b\in X$, $a',b'\in X'$
there are a constant $L_\rho$ (depending only on $\rho$) and
a $(\rho,L_\rho)$-LEL map
$$
 f:(X,d_{X,a,b},\CCc_{X,a,b})\to (X',d_{X',a',b'},\CCc_{X',a',b'})
$$
with $f(a)=a'$ and $f(b)=b'$.
Moreover, $f$ is equal to the composition
$\varphi_{X',a',b'}\circ f_k\circ \psi_{X,a,b}$ with some odd $k\ge 3$.
\end{proposition}

Finally, we prove the next lemma which is a consequence of
Proposition~\ref{P:Sp12a-ThmC}.

\begin{lemma}\label{L:sp12a-arc}
Keeping the notation from Proposition~\ref{P:Sp12a-ThmC},
there is a constant $\gamma>0$ with the following property:
For every dendrite $X$ and every two distinct points $a,b$ of it,
there exists a
$$
 (\gamma,1)-LEL^* \text{ retraction }
 r=r_{X,a,b}:(X,d_{X,a,b},\CCc_{X,a,b}) \to (A,d_A,\CCc_A),
$$
where $A=[a,b]$, $d_A$ is the restriction of $d_{X,a,b}$ onto $A\times A$ and
$\CCc_A$ is the system of all subarcs of $A$.
\end{lemma}
\begin{proof}
Let $\varrho$ be the constant from (c) of Proposition~\ref{P:Sp12a-ThmC}
and put $\gamma=(1-q)^2\varrho/4$; we prove that this $\gamma$
satisfies the above stated property. To this end,
let $X$ be a dendrite, $a\ne b$ be points of $X$ and $A=[a,b]$.
Write $X$ in the form (\ref{EQ:XasUnionXi}) with $A_1=A$.
If $X=A$, the assertion is trivial (take $r=\id_X$); so assume that $X\ne A$.
Proposition~\ref{P:Sp12a-ThmC} gives us a convex metric $d=d_{X,a,b}$
and maps $\varphi_{X,a,b}$, $\psi_{X,a,b}$
satisfying the conditions (a)--(g); put $\CCc=\CCc_{X,a,b}$
and let $r_A:X\to A$ be the first point retraction.
Let $y_k$ ($k\in K$) be all points $y$ of $A$ with non-degenerate preimage
$r_A^{-1}(y)$. Put $Y_k=r_A^{-1}(y_k)$ ($k\in K$); then
$Y_k$'s are pairwise disjoint non-degenerate subdendrites of $X$
and $X=A\cup\bigcup_k Y_k$
Put $t_0=d(a,b)$, $t_k=d(a,y_k)$ ($k\ge 1$)
and for any $t\in [0,t_0]$ denote by $a_t$ the unique point
of $A$ satisfying $d(a_t,a)=t$; hence e.g.~$a_{t_k}=y_k$
and $\lengthd{d}{[a_s,a_t]}=t-s$ for any $s<t$.
Now define $r=r_{X,a,b}:X\to A$ by
$$
 r(x)=\begin{cases}
  x  & \text{if }x\in A;
 \\
  a_{t_k+d(x,y_k)}   & \text{if } x\in Y_k \text{ with } t_k\le t_0/2;
 \\
  a_{t_k-d(x,y_k)}   & \text{if } x\in Y_k \text{ with } t_k> t_0/2.
 \end{cases}
$$
By Proposition~\ref{P:Sp12a-ThmC}(a) and the choice of $q$,
$t_0>2/3$ and so $d(x,y_k)<1/3<t_0/2$ for every $x\in Y_k$.
Hence $r(x)$ is well defined for every $x$. Trivially, $r$ is a continuous
retraction of $X$ onto $A$. Using convexity of $d$ one can easily show
that $\Lip(r)=1$. 

It remains to prove that $r$ is $\gamma$-length expanding$^*$. To this end, fix
$C\in \CCc$, $C=\varphi_{X,a,b}([s_0,s_1])$ and put
$u_i=r_A\circ \varphi_{X,a,b}(s_i)$ ($i=0,1$).
By Proposition~\ref{P:Sp12a-ThmC}(g) we may write
$$
 C=\left[
   C_0
   \sqcup \left(\bigsqcup_{k:\ y_k\in (u_0,u_1)} Y_k \right)
   \sqcup C_1
 \right]
 \cup [u_0,u_1]
$$
with $C_i$ ($i=0,1$)
being connected (either degenerate or belonging to
$\CCc$).
We will assume that there is $k$ with $y_k\in (u_0,u_1)$; the other (simpler)
case can
be described analogously.
Put $n_0=\min\{n:\ A_n\subseteq Y_k,\ u_0<y_k<u_1\} \ge 2$.
Then, using (f) from Proposition~\ref{P:Sp12a-ThmC},
\begin{eqnarray*}
 \lengthd{d}{C}
&=&
 \lengthd{d}{C_0} + \sum_{k:\ y_k\in (u_0,u_1)} \lengthd{d}{Y_k} +
 \lengthd{d}{C_1} + \lengthd{d}{[u_0,u_1]}
\\
&\le&
 \frac{4}{1-q}
 \cdot\max\{
  \lengthd{d}{C_0},
  \lengthd{d}{A_{n_0}},
  \lengthd{d}{C_1},
  \lengthd{d}{[u_0,u_1]}
 \}.
\end{eqnarray*}
On the other hand, by Proposition~\ref{P:Sp12a-ThmC}(d)
we have $\psi(x)=c\cdot d(a,x)$ for every $x\in X$,
where $c=1/d(a,b) \in (1,1/(1-q))$.
For given $k$ and a continuum $Y\subseteq Y_k$ from $\CCc$
containing $y_k$ we have $\psi(Y)=[\psi(y_k),\psi(z_k)]$,
where $z_k$ is such that $d(z_k,y_k)=\max_{y\in Y_k} d(z,y_k)$.
Thus, by convexity of $d$,
$\abs{\psi(Y)}=c\cdot(d(a,z_k)-d(a,y_k))=c \cdot d(y_k,z_k)=c\cdot
\lengthd{d}{r(Y)}$.
By Proposition~\ref{P:Sp12a-ThmC}(c), for every such continuum $Y$ we have
$\lengthd{d}{r(Y)}\ge (\varrho/c)\cdot\lengthd{d}{Y}
\ge \varrho (1-q) \cdot\lengthd{d}{Y}$. Hence, using the fact that
$r|_A$ is the identity,
\begin{eqnarray*}
 \lengthd{d}{r(C)}
 &\ge& \max\{
  \lengthd{d}{r(C_0)},
  \lengthd{d}{r(A_{n_0})},
  \lengthd{d}{r(C_1)},
  \lengthd{d}{r([u_0,u_1])}
 \}
\\
 &\ge&
 \varrho(1-q)  \cdot \max\{
  \lengthd{d}{C_0},
  \lengthd{d}{A_{n_0}},
  \lengthd{d}{C_1},
  \lengthd{d}{[u_0,u_1]}
 \}.
\end{eqnarray*}
Thus, by the choice of $\gamma$,
$\lengthd{d}{r(C)}\ge \gamma\cdot \lengthd{d}{C}$ for every $C\in\CCc$. The
lemma is proved.
\end{proof}

\subsection{$P$-Lipschitz maps}
In \cite{Sp12b} we introduced the so-called $P$-Lipschitz maps (where
$P$ is a finite invariant set) and we gave an upper bound
for their entropy \cite[Proposition~3.3]{Sp12b}.
We recall the definition
as well as the corresponding result here.
If $X$ is a continuum,
a \emph{splitting} of $X$ is any system
$\AAa=\{X_1,\dots,X_n\}$ of non-degenerate subcontinua
covering $X$
such that $P_\AAa=\bigcup_{i\ne j} X_i\cap X_j$ is finite.

\begin{definition}[\cite{Sp12b}]\label{D:P-Lipschitz}
Let $X$ be a non-degenerate continuum and $f:X\to X$ be a continuous map. Let $P$ be a finite $f$-invariant
 subset of $X$, $\AAa$ be a splitting of $X$ with $P_\AAa\subseteq P$
 and $(L_A)_{A\in\AAa}$ be positive constants.
 Then we say that $f$ is
 \emph{$P$-Lipschitz (w.r.t.~the splitting $\AAa$ and the constants $(L_A)_{A\in\AAa}$)}
 if, for every $A\in\AAa$, $f(A)$ is non-degenerate and $\Lip(f|_A)\le L_A$.
\end{definition}

Let $f$ be a $P$-Lipschitz map w.r.t.~$\AAa$.
For $A,B\in\AAa$ we write $A\to B$ or, more
precisely, $A\toPar{f} B$ provided $f(A)$ intersects the interior of
$B$.
The \emph{$P$-transition graph $G_f$ of $f$} is the directed
graph the vertices of which
are the sets $A\in\AAa$ and the edges of which correspond to $A\to B$.
The corresponding $01$-transition matrix will be called the
\emph{$P$-transition matrix} and
will be denoted by $M_f$.
For an integer $n\ge 1$ put
$$
 \AAa^n = \{\AAA=(A_0,A_1,\dots, A_{n-1}):\
 A_i\in \AAa,\ A_0\to A_1\to \dots\to A_{n-1}
 \};
$$
that is, $\AAa^n$ is the set of all paths of length $n-1$ in $G_f$.
Finally, for
non-empty $\BBb\subseteq \AAa$ put
\begin{equation}\label{EQ:thetaB}
  \theta_\BBb
  =\limsup\limits_{n\to\infty} \frac{k_n^{\BBb}}{n}
  \qquad
  \text{where}
  \quad
  k^{\BBb}_n
  =\max\limits_{(A_0,\dots,A_{n-1})\in \AAa^{n}} \card{ \{j: A_j\not\in\BBb}\} \,,
\end{equation}
which is a quantity measuring the maximal ``asymptotic frequency'' of occurrences of
$A\in \AAa\setminus\BBb$ in paths of $f$. The following proposition gives an upper bound for
the entropy of $P$-Lipschitz maps. There and below, $\log^+ x$ denotes $\max\{\log x,0\}$.

\begin{proposition}[\cite{Sp12b}]\label{P:entropyOfPLipschitz}
 Let $X$ be a non-degenerate totally regular continuum endowed with a convex metric $d$
 such that $\lengthd{d}{X}<\infty$.
 Let $f:X\to X$ be $P$-Lipschitz w.r.t.~$\AAa$ and $(L_A)_{A\in\AAa}$.
 Then, for every non-empty subsystem $\BBb$ of $\AAa$,
 $$
  h(f) \le \log^+ L_\BBb + 2\theta_\BBb \log^+ L_\AAa,
 $$
 where $L_{\BBb} = \max_{A\in\BBb} L_A$ and $L_{\AAa} = \max_{A\in\AAa} L_A$.
\end{proposition}

\section{Markov maps on trees}\label{S:markov}
Let $T$ be a tree and $f:T\to T$ be a continuous
map. Assume that a finite set $P\subset T$ is $f$-invariant
and contains all vertices of $T$. We say that $f$ is \emph{$P$-Markov}
if it is monotone on each component of $T\setminus P$.
The closures of
the components of $T\setminus P$ are called \emph{$P$-basic arcs}
of $f$.
The \emph{$P$-transition graph} of $f$ with respect to $P$
is the directed graph having $P$-basic arcs as vertices,
with an edge from $A$ to $B$ (we write $A\to B$)
if and only if $B\subseteq f(A)$. The corresponding $01$-matrix
will be denoted by $M_f$ and called the \emph{$P$-transition matrix}.
For the following results see e.g.~\cite[Proposition~1.4, Corollary~1.11]{Bal01}.

\begin{lemma}\label{L:entropyMarkov} Let $T$ be a tree and let
$f:T\to T$ be a $P$-Markov map with the transition matrix $M$
and the Perron-Frobenius value $\lambda_M$. Then
$h(f) = \log^+ \lambda_M$.
\end{lemma}

\begin{lemma}\label{L:transitiveMarkov} Let $T$ be a tree and let
$f:T\to T$ be a $P$-Markov map with the transition matrix $M$. Then
\begin{enumerate}
	\item[(a)] $f$ is transitive if and only if $M$ is irreducible
	  and not a permutation matrix;
	\item[(b)] $f$ is exact if and only if $M$ is primitive.
\end{enumerate}
\end{lemma}


A special class of Markov maps consists of the so-called
$P$-linear Markov maps. 
Let $T$ be a tree endowed with a convex metric $d$. Let
$f:T\to T$ be a $P$-Markov map.
We say that $f$ is \emph{$P$-linear} if for every component $C$ of
$T\setminus P$ there is a positive constant $\lambda_C$ such that
$$
 d(f(x), f(y)) = \lambda_C \cdot d(x,y)
 \qquad\text{for every } x,y\in C.
$$
Moreover, if there is $\lambda$ such that every $\lambda_C$ is equal to $\lambda$,
we say that
$f$ is \emph{$P$-linear with constant slope $\lambda$}.

By \cite{BC01}, every
$P$-Markov tree map $f:T\to T$ is semiconjugate to a $P'$-linear tree map
$f':T'\to T'$ with
constant slope $\lambda'=e^{h(f)}$. The following trivial fact shows that
if $f$ is $P$-linear and transitive
then the ``constant slope'' can be achieved by a simple change of the metric on $T$.

\begin{lemma}\label{L:parryMapOnTree}
 Let $T$ be a tree with a convex metric $d$
 and let $f:T\to T$ be a transitive $P$-linear Markov map on $T$ with the entropy
 $h(f)=\log\lambda$.
 Then there is an equivalent convex metric $d'$ on $T$ such that
 $f:(T,d')\to (T,d')$ is a $P$-linear Markov map with constant slope $\lambda$.
\end{lemma}
\begin{proof} 
 Denote the $P$-basic arcs by $A_i=[a_i,b_i]$, $i=1,\dots,m$. Since
 $T$ is uniquely arcwise connected, every $A_i$ is geodesic and so
 its length is $\lengthd{d}{A_i}=d(a_i,b_i)$.
 Since $f$ is transitive, its $P$-transition matrix $M$
 is irreducible. Thus $\lambda$ is the Perron eigenvalue of
 $M$ with a strictly positive right eigenvector $\alpha=(\alpha_1,\dots,\alpha_m)$;
 that is, $M\cdot\alpha = \lambda \alpha$.
 Let $d'$ be the unique convex metric on $T$
 such that
 $$
  d'(x,y)=\frac{\alpha_i}{d(a_i,b_i)}  \cdot d(x,y)
  \qquad
  \text{for every } i \text{ and } x,y\in A_i.
 $$
 Trivially $f$ is $P'$-linear on $(T,d')$ and
 $\lengthd{d'}{A_i}=\alpha_i$ for every $i$.
 By the choice of $\alpha$ we have that
 $$
  \lambda\cdot \lengthd{d'}{A_i} = \sum\limits_{j: A_i\to A_j}
  \lengthd{d'}{A_j}=\lengthd{d'}{f(A_i)}
 $$
 for every $i$.
 Hence $f$ has constant slope $\lambda$ on $(T,d')$.
\end{proof}

\subsection{$(P,S)$-linear Markov maps}\label{SS:specialMarkovMap}
Here we introduce the so-called $(P,S)$-linear Markov maps, where
$S=(s_0,s_1,\dots,s_n)$ is a tuple of points from $P$; see Definition~\ref{D:specialMarkovMap}.
Maps with this property can be used to define exactly Devaney chaotic tree maps
with ``controlled'' entropy, see Proposition~\ref{P:exactExtensionToTree}.

\begin{definition}\label{D:specialMarkovMap}
Let $T$ be a tree with a convex metric, $P\subseteq T$ be finite
and $f:T\to T$ be a $P$-linear Markov map.
Let $n\ge 1$ and $S=(s_0,s_1,\dots,s_n)$ be a tuple of distinct points from $P$.
We say that $f$ is \emph{$(P,S)$-linear} if
\begin{enumerate}
	\item[(a)] $f(s_i)=s_{i+1}$ for every $0\le i <n$;
	\item[(b)] $A_S=[s_0,s_n]$ is a $P$-basic arc and $s_n$ is an end point of $T$;
	\item[(c)] for every $P$-basic arc $A$ different from $A_S$
	  there is a path from $A$ to $A_S$.
\end{enumerate}
\end{definition}

Tree maps with this property were constructed e.g.~in
the proof of \cite[Theorem~1.2]{AKLS} and also in the proofs
of \cite[Lemmas~4.6,4.7]{Ye}. Let us formulate the latter fact as a lemma.
(For Ye's classification of trees, see Subsection~\ref{SS:trees}.)

\begin{lemma}\label{L:YeIsPSLinear}
 Let $T$ be a tree and let $i$ be such that $T\in\TTt_i$. Then, for every $\eps>0$,
 there are a finite set $P$, a tuple $S=(s_0,\dots,s_m)$ with
 $\{s_1,\dots,s_m\}\subseteq \Endpoints(T)$
 and a transitive $(P,S)$-linear map $f:T\to T$ with
 $$
  h(f) < \frac{\log 2}{m} + \eps,
  \qquad
  \text{where } m=\card\Endpoints(T)-i.
 $$
\end{lemma}
\begin{proof}
We will keep the notation from the proofs of \cite[Lemmas~4.6, 4.7]{Ye}.
Let $T\in\TTt_i^{n_1,\dots,n_r}$, where $i\ge 0$, $r\ge 1$ and $n_1,\dots,n_r\ge 2$.
Put $m=m_0=\prod_{l=1}^r n_l = \card \Endpoints(T)-i$
and $n_0=r+i$.
The end points of $T$ are $t_q^1$ ($1\le q\le m_0$) and
$y^j$ ($1\le j \le i$). The maps $f_n=f_{i,r,n}$ constructed
in the proofs are $P_n$-linear Markov maps (we put $P_n=Q_n$ if $i\ge 1$).
If we put $S_n=(s_0,\dots,s_{m_0})$
with $s_q=t_q^1$ for $q=1,\dots,m_0$, and $s_0=x^1_{(n+n_0)m_0}$,
then the first two conditions from Definition~\ref{D:specialMarkovMap}
are satisfied, see \cite[p.~301]{Ye}.
The third one is satisfied trivially since the maps $f_n$ are transitive.
Thus every $f_n$ is $(P_n,S_n)$-linear.
Since $\lim h(f_n)=(1/m_0)\log 2$, the proof is finished.
\end{proof}

Let us note that in the proofs of Propositions~\ref{P:exactExtensionToTree}
and \ref{P:thmA} we do not need to assume that
the considered $(P,S)$-linear maps are transitive. In Appendix~1 we give simple constructions of
(non-transitive) zero entropy $(P,S)$-linear Markov maps on stars and combs; this makes
the proof of Theorem~\ref{T:main} not dependent on the results of \cite{Ye}.

\subsection{Exact $P$-linear Markov maps}
In Proposition~\ref{P:exactExtensionToTree} we show  that every $(P,S)$-linear Markov
map $f$ on a tree $R$ can be used to define an exact map $g$ on a tree
$T\supseteq R$ such that the entropy of $g$ is small provided the entropy of $f$ is small
and the cardinality of $S$ is large.
Such maps $g$ will be used in Section~\ref{S:exactOnDendrites}
in the construction of exact small entropy maps
on dendrites, see Proposition~\ref{P:thmA}.
First we introduce some notation which will be used for the entropy estimation.

Let $\PPp$ denote the system of maps of two variables
$x,N$ of the type
$$
 \psi:(1,\infty)\times\NNN \to \RRR,
 \quad
 \psi(x,N)
 = \frac{b_{1,N}}{x} + \frac{b_{2,N}}{x^2} + \dots + \frac{b_{i,N}}{x^i} + \dots
$$
such that for some $b=b_\psi>0$ and $k=k_\psi\in\NNN$ it holds that
\begin{equation}\label{EQ:functionsFromPPp}
 \abs{b_{i,N}} \le b\cdot N^k \cdot i^k
 \qquad\text{for every } i,N\in\NNN.
\end{equation}
Notice that (\ref{EQ:functionsFromPPp}) guarantees that $\psi(x,N)$ is always
finite, since
$$
 \abs{\psi(x,N)}
 \le
 b\cdot N^k \cdot \nu_k(x),
 \qquad\text{where}
 \quad
 \nu_k(x)=\sum_{i=1}^\infty \frac{i^k}{x^i}<\infty.
$$

\begin{lemma}\label{L:systemP}
The system $\PPp$ is closed under multiplication and linear combinations,
that is, $c\psi + c'\psi'\in\PPp$ and $\psi\cdot \psi'\in\PPp$
for every
$\psi,\psi'\in \PPp$ and  $c,c'\in\RRR$.
\end{lemma}
\begin{proof}
Let $c,c'\in\RRR$, $\psi,\psi'\in \PPp$ and  let
$b_{i,N}, b'_{i,N}$ be the coefficients of $\psi,\psi'$, respectively.
Put $b=b_\psi$, $k=k_\psi$, $b'=b_{\psi'}$ and $k'=k_{\psi'}$.

The fact that $c\psi + c'\psi'\in\PPp$ is trivial.
To show that $\varphi=\psi\cdot \psi'$ belongs to $\PPp$, realize that
$\varphi(x,N)=\sum_{i=2}^\infty c_{i,N}/x^i$, where
the coefficients $c_{i,N}$ satisfy
$$
 \abs{c_{i,N}}
 =
 \left\lvert
   \sum_{j=1}^{i-1} b_{j,N}\cdot b'_{i-j,N}
 \right\rvert
 \le b\cdot b'\cdot N^{k+k'} \cdot\sum_{j=1}^{i-1} j^k\cdot (i-j)^{k'} .
$$
The sum in the right-hand side can be bounded from above by
$i^{k+k'}\cdot(i-1) < i^{k+k'+1}$.
So we can take $b_{\varphi}=b\cdot b'$ and $k_{\varphi}=k+k'+1$.
\end{proof}

Denote by $\Poo$ the subsystem of $\PPp$ consisting of those $\psi\in\PPp$ for which
there is $l\in \NNN$ (not depending on $N$) with
$b_{1,N}=b_{2,N}=\dots=b_{N-l,N}=0$ for every $N\ge l$.

\begin{lemma}\label{L:systemPo}
The following are true:
\begin{enumerate}
	\item[(a)] $\psi\in\Poo$ if and only if there is $\varphi\in\PPp$
	and $l\in\NNN$ such that
		$$ 
		 \psi(x,N)=\frac{1}{x^{N-l}} \cdot \varphi(x,N).
		$$ 

	\item[(b)] If $\delta>0$ and $\psi\in\Poo$ then, for $N\to\infty$,
		 $x^{N/2}\cdot \psi(x,N)$ converges to $0$ uniformly on
     $[1+\delta,\infty)$.

	\item[(c)]
    $c\psi + c'\psi'\in\Poo$  and $\psi\cdot \psi''\in\Poo$
   for every $\psi,\psi'\in\Poo$, $\psi''\in\PPp$ and $c,c'\in\RRR$.
\end{enumerate}
\end{lemma}
\begin{proof}
(a) If $\psi\in\Poo$, there is $l\in\NNN$ such that $b_{i,N}=0$ for every $N$ and
every $i\le N-l$. Put $\varphi(x,N)=\sum_i c_{i,N} x^{-i}$, where $c_{i,N}=b_{N-l+i,N}$.
Then $\psi(x,N)=x^{-(N-l)} \varphi(x,N)$ and
$\abs{c_{i,N}}\le b\cdot N^k \cdot (N-l+i)^k \le b\cdot N^{2k}\cdot i^k$,
so $\varphi\in\PPp$. The reverse implication is obvious.

(b) follows from the inequality
$
 \abs{x^{N/2}\cdot \psi(x,N)}
 \le
 {x^{-(N/2-l)}} \cdot b \cdot N^k \cdot \nu_k(x)
$,
where $k=k_{\varphi}$, $b=b_{\varphi}$ with $\varphi$ is given by (a)
and $\nu_k(x)=\sum_{i=1}^\infty i^k x^{-i}$.

(c) follows from (a) and Lemma~\ref{L:systemP}.
\end{proof}

\begin{proposition}\label{P:exactExtensionToTree}
Let $R$ be a tree and let $f:R\to R$
be a $(P,S)$-linear map, where $S=(s_0,s_1,\dots,s_n)$
with $n\ge 2$. Let $T\supseteq R$ be a tree such that
\begin{equation}\label{EQ:T'=TcupNArcs}
 T=R\sqcup \bigsqcup_{i=1}^n (s_i,t_i].
\end{equation}
Then for every $\eps>0$
there is a map $g:T\to T$ such that
the following hold:
\begin{enumerate}
	\item[(a)] $g$ is a $Q$-linear Markov map, where $Q\cap R=P$;
	\item[(b)] $g$ is exactly Devaney chaotic;
	\item[(c)] $(1/n)\log 2-\eps < h(g) < \max\{h(f), (1/n)\log{2}\}+\eps$;
	\item[(d)] $g(x)=f(x)$ for every $x\in R\setminus (s_0,s_n]$.
\end{enumerate}
\end{proposition}

\begin{proof} Endow $T$ with a convex metric.
For $i=1,\dots,n$ put $A_i=[s_i,t_i]$ and define an order on $A_i$ in
such a way that $s_i<t_i$. Denote the $P$-basic arcs of $f$ by $B_1,\dots,B_p$, where
$B_p=[s_0,s_n]$.
Let $\varphi:I\to R$ be a continuous surjection such that $\varphi(0)=\varphi(1)=s_1$
and, for some $m\in\NNN$, $\varphi$ maps linearly every $[(l-1)/m, l/m]$
onto some $B_{j_l}$.

Fix an integer $N>6$. In every $[s_i,t_i]$ choose points $t_i^j$
($j=0,1,\dots,N$) in such a way that
$s_i=t_i^0<t_i^1<\dots<t_i^N=t_i$; for $j=1,\dots,N$ put
$A_i^j=[t_i^{j-1},t_i^j]$. Moreover, take points
$t_n^{N-2}=t_n^{N-1,0}<t_n^{N-1,1}<\dots < t_n^{N-1,m}=t_n^{N-1}$
in $A_n^{N-1}$ and put $A_n^{N-1,l}=[t_n^{N-1,l-1}, t_n^{N-1,l}]$
for $l=1,\dots,m$. 

\begin{figure}[ht!]
  \includegraphics{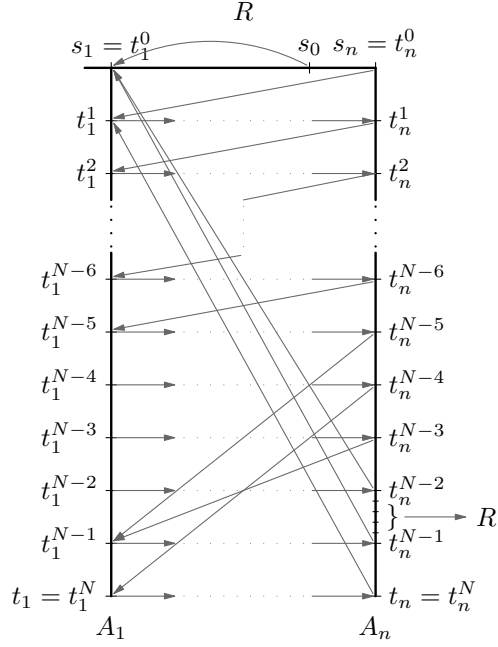}
  \caption{The map $g=g_N:T\to T$ (here $R$ is an arc)}
  \label{Fig:P:exactExtensionToTree-map}
\end{figure}

Put
$$
 Q= P
 \sqcup \{t_i^j:\ 1\le i\le n,\ 1\le j\le N\}
 \sqcup \{t_n^{N-1,l}:\ 1\le l <m\}
$$
and define a map $g=g_N:T\to T$ as follows (see Figure~\ref{Fig:P:exactExtensionToTree-map}):
\begin{itemize}
	\item $g(x)=f(x)$ for $x\in R\setminus (s_0,s_n]$;
	\item $g(t_i^j)=t_{i+1}^{j}$ for every $i<n$ and $j\ge 1$;
	\item $g(t_n^j)=
	  \begin{cases}
	   t_1^{j+1}  &\text{if } j=0,1,\dots,N-6;
	  \\
	   t_1^{N-1}  &\text{if } j=N-5,N-3;	
	  \\
	   t_1^{N}  &\text{if } j=N-4;	
	  \\
	   t_1^{0}=s_1  &\text{if } j=N-2, N-1;	
	  \\
	   t_1^{1}  &\text{if } j=N;	
	  \end{cases}
	  $
	  \\
	  in particular, $g(s_n)=g(t_n^0)=t_1^1$;
	\item  $g(t_n^{N-1,l})=\varphi(l/m)$ for every $1\le l < m$;
	\item  $g$ is $Q$-linear.
\end{itemize}
Obviously, the map $g$ satisfies (d), it is a $Q$-linear Markov map and $Q\cap R=P$,
so (a) is also satisfied.
To prove (b) and (c) we need to describe the transition graph of $g$.
The $Q$-basic arcs are:
\begin{equation}\label{EQ:Q-basicArcs}
\begin{split}
 B_i \quad (1\le i \le p),
 \quad
 &A_i^j \quad (1\le i \le n, 1\le j\le N, (i,j)\ne (n,N-1))
 \\
 \text{and}\quad
 &A_n^{N-1,l} \quad (1\le l\le m).
\end{split}
\end{equation}
The transition graph $G_g$ of $g$ is as follows (see
Figure~\ref{Fig:P:exactExtensionToTree-graph}):
\begin{itemize}
	\item $B_i\tof{g} B_j$ if and only if $B_i\tof{f} B_j$ for every $i<p$ and every $j$;
	\item $B_p\tof{g} A_1^1$;
	\item $A_i^j\tof{g} A_{i+1}^j$ if $i<n$ and $(i,j)\ne (n-1,N-1)$;
	\item $A_{n-1}^{N-1}\tof{g} A_n^{N-1,l}$ for every $l=1,\dots,m$;
	\item $A_n^j \tof{g}
	  \begin{cases}
	   A_1^{j+1}  &\text{if } j\le N-6;
	  \\
	   A_1^{N-k} \ (k=1,\dots,4)  &\text{if } j= N-5;
	  \\
	   A_1^{N}    &\text{if } j= N-4,N-3;
	  \\
	   A_1^{k} \ (k=1,\dots,N-1)  &\text{if } j= N-2;
	  \\
	   A_1^1  &\text{if } j= N;
	  \end{cases}$
	\item $A_n^{N-1,l}\tof{g} B_{j_l}$ for every $l=1,\dots,m$.
\end{itemize}

\begin{figure}[ht!]
  \includegraphics{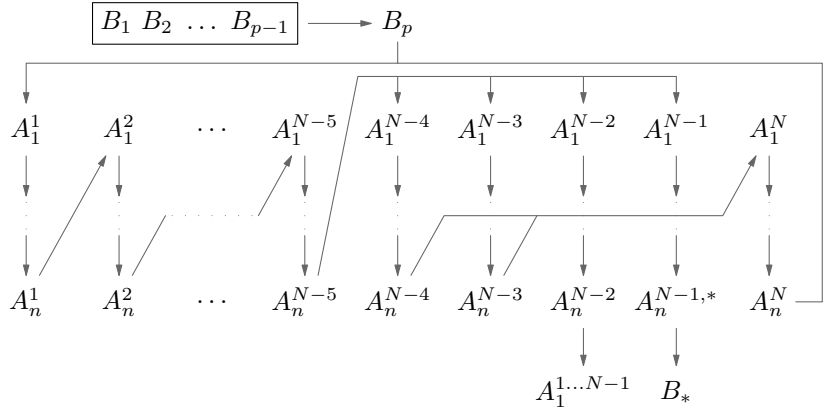}
  \caption{The transition graph of $g$}
  \label{Fig:P:exactExtensionToTree-graph}
\end{figure}

Now we prove (b).
Since there is a path from every $B_i$ to $B_p$ ($f$ is $(P,S)$-linear),
$\varphi$
is a surjection (so $\{j_1,\dots,j_m\}=\{1,\dots,p\}$)
and there is a path from $B_p$ to any $A_i^j, A_n^{N-1,l}$,
we have that the transition graph $G_g$ is irreducible.
But we even have that $G_g$ is primitive since there
are two loops over $B_p$ the lengths of which differ by $1$.
Indeed, take a loop through some $A_n^{N-1,l}$ such that $j_l=p$;
then take $l'$ such that $B_{j_{l'}}\tof{g} B_p$, $j_{l'}<p$ (notice that
$p\ge 2$ since $n\ge 2$) and
construct another loop by replacing the subpath
$A_{n-1}^{N-1}\tof{g} A_n^{N-1,l} \tof{g} B_p$
by $A_{n-1}^{N-1}\tof{g} A_n^{N-1,l'} \tof{g} B_{j_{l'}} \tof{g} B_p$.
So $g$ is an exact map; since on trees transitivity implies dense periodicity
we have that $g$ is exactly Devaney chaotic.

\medskip

The rest of the proof is devoted to (c).
Recall that $g=g_N$ depends on $N$. To prove the upper bound in (c)
it suffices to
show that
$\limsup_{N\to\infty} h(g_N) \le  \max\{h(f), (1/n)\log 2\}$.
Let $\lambda_f$ be the Perron eigenvalue of the transition matrix of $f$.
Let $M_{g_N}$ be the transition matrix of ${g_N}$ and
$\lambda_{g_N}$ be the Perron eigenvalue of $M_{g_N}$.
Since $M_{g_N}$ is primitive we have that $\lambda_{g_N}>1$.
Let $M_N$ be the $01$-matrix created from $M_{g_N}$ by replacing $0$'s
with $1$'s in the positions corresponding to the edges from $B_p$ to $B_i$
provided $B_p\tof{f} B_i$. If $\lambda_N$ denotes the Perron eigenvalue of $M_N$
then, since $M_{g_N}\le M_N$,
\begin{equation}\label{EQ:entropy_gN}
 h(g_N)
 = \log \lambda_{g_N}
 \le \log \lambda_N.
\end{equation}

Since every loop in $G_g$ which is not a subset of $\{B_1,\dots,B_p\}$
passes $A_n^1$ or $A_n^{N-2}$, the set
$$
 \RRr=\{B_1,\dots,B_p;\ B_{p+1}=A_n^1, B_{p+2}=A_n^{N-2}\}
$$
of $Q$-basic arcs is a rome for $M_N$. Notice that trivially $\{B_1,\dots,B_p\}$
is a rome for the transition matrix $M_f$ of $f$.
For $i,j=1,\dots,p+2$ put
$$
 r_{ij}
 =r_{ij}(x,N)
 = \sum\limits_{\pi}
  \frac{1}{x^{\abs{\pi}}}
$$
where the summation is over all $\RRr$-simple paths $\pi$ from $B_i$ to $B_j$.
Since for $i,j\le p$ any simple
path from $B_i$ to $B_j$ visits only
arcs $B_1,\dots,B_p$, for $i,j\le p$ the quantity $r_{ij}$ does not depend on $N$.
Let $R_N(x)$ and $R_f(x)$ denote the matrices
$(r_{ij}(x,N))_{i,j=1}^{p+2}$ and $(r_{ij}(x))_{i,j=1}^{p}$, respectively.
Define also
$$
 \Phi_N(x)=\det (R_N(x)-E)
 \qquad\text{and}\qquad
 \Phi_f(x)=\det (R_f(x)-E)
$$
where $E$ denotes the identity matrix of the corresponding dimension.
Then $\lambda_N$ and $\lambda_f$ are the largest positive roots
of $\Phi_N$ and $\Phi_f$, respectively (see Subsection~\ref{SS:PerronFrobenius}).
Our aim is to show that
\begin{equation}\label{EQ:limsup-lambda-N}
 \limsup_{N\to\infty} \lambda_N
 \le
 \max\left\{\lambda_f,\sqrt[n]{2}\right\}.
\end{equation}
To this end we must first estimate $r_{ij}=r_{ij}(x,N)$.

Realize that
\begin{equation}\label{EQ:rij:i<=p<j}
 r_{i,p+1}=0 \quad\text{for } i<p,
 \qquad
 r_{p,p+1}=x^{-n}
 \qquad\text{and}\qquad
 r_{i,p+2}=0 \quad\text{for } i\le p,
\end{equation}
since every path from $B_i$ ($i\le p$) to $B_{p+1}=A_n^1$ or to $B_{p+2}=A_n^{N-1}$
must pass both $B_p$ and $B_{p+1}$; hence the only such path which is simple is
the path
$B_p\to A_1^1 \to A_2^1\dots\to A_n^1=B_{p+1}$ and the length of it is $n$.

The simple paths starting from $B_{p+1}=A_n^1$ to any member of $\RRr$
are ``long'' (their lengths are at least $n(N-5)$) and
their number does not depend on $N$.
Indeed, there are $m$ simple paths from $B_{p+1}$ to some of $B_i$'s ($i\le p$),
two simple loops over $B_{p+1}$ and one simple path from $B_{p+1}$ to $B_{p+2}$.
Hence
\begin{equation}\label{EQ:rij:i=p+1}
 r_{p+1,j}\in\Poo \quad\text{for every } j=1,\dots,p+2.
\end{equation}

There are exactly $(N-3)$ simple loops over $B_{p+2}$; indeed, for any
$j=2,\dots,N-2$ there is exactly one which starts with
$B_{p+2} \to A_1^j$ and
the length of it is $n(N-1-j)$. Hence, for $x>1$,
\begin{equation}\label{EQ:rij:i=p+2=j}
 r_{p+2,p+2}
 =\frac{1}{x^{n}} + \frac{1}{x^{2n}} + \dots +\frac{1}{x^{(N-3) n}}
 =\frac{1}{x^n-1} + \oo,
\end{equation}
where $o\in\Poo$.
Finally,
\begin{equation}\label{EQ:rij:i=p+2>j}
 r_{p+2,i}\in\PPp \quad\text{for every } i\le p+1.
\end{equation}

Now we are ready to estimate $\Phi_N(x)=\det(R_N(x)-E)$. Expanding the determinant
along the $(p+1)$-st row, then along
the last column and using
(\ref{EQ:rij:i=p+1}), (\ref{EQ:rij:i<=p<j}), (\ref{EQ:rij:i=p+2=j}),
(\ref{EQ:rij:i=p+2>j}) and Lemma~\ref{L:systemPo} give
(with $\oo,\oo',\oo''\in\Poo$)
$$
 \Phi_N(x)=(-1)
 \cdot
 \left[ -1 + \frac{1}{x^n-1} + \oo  \right]
 \cdot \det(R_f(x)-E) + \oo'
=
 \frac{x^n-2}{x^n-1} \cdot \Phi_f(x) + \oo''.
$$

Recall that $\lambda_N$ denotes the largest positive root of $\Phi_N(x)$.
Put $\lambda=\limsup \lambda_N$; without loss of generality we may assume
that $\lambda_N\to \lambda$ for $N\to\infty$. We want to show that
$\lambda\le \tilde{\lambda}:=\max\{\lambda_f, \sqrt[n]{2}\}$. Assume, on the contrary,
that $\lambda>\tilde{\lambda}$; so there is $N_0\in\NNN$ and $\delta>0$
such that $\lambda_N\ge \tilde{\lambda}+\delta$ for every $N\ge N_0$.
Put $\psi(x,N)=\Phi_N(x)-\frac{x^n-2}{x^n-1}\cdot \Phi_f(x)$ and take $\eps>0$.
Since $\psi\in\Poo$, by Lemma~\ref{L:systemPo}
there is $N_1\in\NNN$ such that
\begin{equation}\label{EQ:psi_xN-unifConverg}
 x^{N/2}\cdot \abs{\psi(x,N)} <\eps
 \qquad\text{for every } N\ge N_1
 \text{ and every } x\ge \tilde{\lambda}.
\end{equation}
On the other hand,
$\Phi_f(x)\ne 0$ for $x>\lambda_f$; since $x^p \Phi_f(x)$ is a polynomial by
(\ref{EQ:romeMethod}), there is $c>0$ with $x^p \abs{\Phi_f(x)}\ge c$
for every $x\ge\tilde{\lambda}+\delta$.
Thus for every $N\ge N_0$ we have
$$
  \lambda_N^{N/2}\cdot \abs{\psi(\lambda_N,N)}
 = \lambda_N^{N/2}\cdot
   \left\lvert
     0
     -  \frac{\lambda_N^n-2}{\lambda_N^n-1}\cdot \Phi_f(\lambda_N)
   \right\rvert
 \ge
   c\lambda_N^{N/2-p}
   \cdot \frac{\lambda_N^n-2}{\lambda_N^n-1},
$$
which goes to infinity for $N\to\infty$;
this contradicts (\ref{EQ:psi_xN-unifConverg}).
Hence we have proved (\ref{EQ:limsup-lambda-N}).
Since $h(f)=\log^+ \lambda_f$,
the upper bound in (c) for $g=g_N$ with large $N$
follows from (\ref{EQ:entropy_gN}).

To prove the lower bound in (c), let $\tilde{M}_N\le M_{g_N}$ be
the $01$-matrix obtained from $M_{g_N}$ by keeping only those $1$'s
which correspond to paths $A\to B$ with $A,B\in\{A_i^j:\ i\le n,j\le N-2\}$;
all other entries of $M_{g_N}$ are replaced by $0$.
Then $\tilde\RRr=\{A_n^{N-2}\}$ is a rome for $\tilde{M}_N$ and the corresponding
matrix function is
$$
 R_{\tilde\RRr} = (\tilde{r}_{11}), \qquad
 \tilde{r}_{11}(x,N) = \frac{1}{x^{n}} + \frac{1}{x^{2n}} + \dots + \frac{1}{x^{(N-4)n}}\,.
$$
By a simple analysis we can show that the largest positive root $\tilde{\lambda}_N$
of $\det(R_{\tilde{\RRr}}(x)-E)=\tilde{r}_{11}(x)-1$ converges to $\sqrt[n]{2}$.
Thus $\liminf \lambda_{g_N}\ge \lim \tilde{\lambda}_{N} = \sqrt[n]{2}$.
Now also the lower bound in (c) for $g=g_N$ with large $N$ follows.
\end{proof}


\section{Exact small entropy maps on dendrites}\label{S:exactOnDendrites}

Here we prove the main results of the paper stated in the introduction.
They will follow from the next proposition.

\begin{proposition}\label{P:thmA}
 Let $X$ be a non-degenerate dendrite and let $T\subseteq X$ be a tree with end points
 $t_1,\dots,t_n$ and with every terminal
 edge having non-empty interior in $X$. If there exists
 a map $f:T\to T$ which is $(P,S)$-linear for some $P$ and $S=(s_0,t_1,\dots,t_n)$
 (where $s_0\in P$), then
 $$
  \IED(X)\le \max\{h(f), (1/n)\log 2\}.
 $$
\end{proposition}

Before going into the details of the proof, let us outline the main steps of it.
We write $T$ in the form $T = R\sqcup \bigsqcup_{i=1}^n (s_i,t_i]$,
where $R$ is homeomorphic to $T$,
arcs $[s_i,t_i]$ have non-empty interiors
and $s_i,t_i$ are end points of $R,T$, respectively.
We may assume that the $(P,S)$-linear map $f$ is defined on $R$
and $S=(s_0,s_1,\dots,s_n)$ for some $s_0\in P$.
Using Proposition~\ref{P:exactExtensionToTree},
we construct an exactly Devaney chaotic $Q$-linear map $g:T\to T$
with $h(g)<\max\{h(f),(1/n)\log 2\}+\eps$.

Using the $Q$-basic arcs $A$ of $T$ we construct a splitting $\tilde{\AAa}$
of $X$ into subcontinua $\tilde{A}$ in such a way that most of
the members of $\tilde{\AAa}$ are free arcs; only subcontinua $\tilde{A}$ corresponding
to $Q$-basic arcs $A$ for which either $A\subseteq R$ or $A\ni t_i$,
may be ``complicated'' (that is, they need not be free arcs).

Then we construct a map $F:X\to X$ in three steps, see (\ref{EQ:thmA:def_of_F}).
First, on most of the subcontinua $\tilde{A}$
which are free arcs we put $F(x)=g(x)$. Second,
for subcontinua $\tilde{A}$ which either contain an end point of $T$ or are
not subsets of $R$ and $g(A)$ is a subset of some ``complicated'' $\tilde{B}$
(see the definition of $\DDd$ in (\ref{EQ:thmA:defOfCCc})), we
use Propositions~\ref{P:Sp12a-ThmC} and \ref{P:Sp12a-ThmD} to define $F$ in such a way
that the restrictions $F|_{\tilde{A}}$ are LEL-maps.
Finally, on the remaining ``complicated'' subcontinua $\tilde{A}$, (that is, for those
$\tilde{A}$ with $A$ being a $Q$-basic arc of $R$),
we define $F$ to be the composition of $g$ with a
Lipschitz-1 retraction of $\tilde{A}$ onto $A$, obtained from Lemma~\ref{L:sp12a-arc}.

During the construction of $F$ we also define a convex metric $d$ on $X$
(see (\ref{EQ:thmA:def_d})), mainly
using Lemma~\ref{L:parryMapOnTree}
(to obtain a convex metric $d_T$ on $T$ such that $g$ has constant slope)
and Proposition~\ref{P:Sp12a-ThmC} (to obtain a convex metric on ``complicated'' subcontinua
$\tilde{A}$). Then we prove that
the map $F$ is $Q$-Lipschitz w.r.t.~this metric and the splitting $\tilde{\AAa}$,
see (\ref{EQ:thmA:F-is-Q-Lipschitz}). Using properties of $g$,
mainly the inequality (\ref{EQ:theta_B}),
and the fact that
the transition matrices $M_F$ and $M_g$ of $F$ and $g$ coincide,
Proposition~\ref{P:entropyOfPLipschitz}
gives us that the entropy of $F$ can be only slightly larger than that
of $g$, see (\ref{EQ:thmA:entropyFN}).
Finally, we prove that $F$ is exactly Devaney chaotic.

\begin{proof} 
Let $X$ be a dendrite and let $T\subseteq X$ be a tree with
end points $t_1,\dots,t_n$ and with
every terminal edge having non-empty interior in $X$.
So we may write
$$
 T = R\sqcup \bigsqcup_{i=1}^n (s_i,t_i],
$$
where $R$ is homeomorphic to $T$,
$\Endpoints(R)=\{s_1\dots,s_n\}$ and
for every $i$ there is $u_i\in (s_i,t_i)$ such that $[s_i,u_i]$ is a free arc in $X$.
By the assumption
there are $P\subseteq R$ and a $(P,S)$-linear Markov map $f:R\to R$,
where $S=(s_0,s_{\pi(1)},\dots,s_{\pi(n)})$ for some $s_0\in R$
and a permutation $\pi$ of $\{1,\dots,n\}$.
By reordering the indices, if necessary, we may
assume that $\pi$ is the identity, that is, $S=(s_0,s_1,\dots,s_n)$.

Let $N>6$, let $g=g_N:T\to T$ be the map
constructed in the proof of Proposition~\ref{P:exactExtensionToTree}
and $\lambda=\lambda_N$ be such that $h(g)=\log\lambda$; then
\begin{equation}\label{EQ:thmA:limsup_lambdaN}
 \limsup_{N\to\infty} \lambda_N
 \le
 \exp \left(  \max\{h(f), (1/n)\log 2\}  \right).
\end{equation}
We will keep the notation from the proof of Proposition~\ref{P:exactExtensionToTree}
(see e.g.~Figures~\ref{Fig:P:exactExtensionToTree-map} and
\ref{Fig:P:exactExtensionToTree-graph});
without loss of generality we may assume that $u_i=t_i^{N-1}$ for every $i$.
Let $\AAa$ be the set of all $Q$-basic arcs of $g$, see (\ref{EQ:Q-basicArcs}).
Put
\begin{equation}\label{EQ:thmA:defOfCCc}
 \DDd = \{A_i^N,A_n^{N-4},A_n^{N-3}, A_n^{N-1,l}:\ i\le n, l\le m\},\qquad
 \BBb=\AAa\setminus\DDd.
\end{equation}
We first prove that
\begin{equation}\label{EQ:theta_B}
 \theta_\BBb\le \frac{2}{N-5}
\end{equation}
where $\theta_{\BBb} = \limsup_{h\to\infty} ({k_h}/{h})$,
$k_h
 =\max_\CCC \card\{i<h:\ C_i\not\in\BBb\}$
 and the maximum is taken over all ${\CCC=(C_0,\dots,C_{h-1})\in\AAa^h}$.
To this end take $h\in\NNN$ and $(C_0,\dots,C_{h-1})\in\AAa^h$.
If $i<h$ is such that $C_i\in\DDd$
then there are $l\in\{i,\dots, i+n\}$ and
$l'\ge \min\{h-1,l+n(N-5)\}$ such that
$$
 C_{j}\in\BBb
 \qquad\text{ for every }
 l<j\le l'.
$$
(Indeed, to define $l$, distinguish two cases:
if $g(C_i)\subseteq R$ put $l=i$; otherwise
$g^{n'}(C_i)=A_1^1$ for some $1\le n'\le n+1$ and put $l=i+n'-1$.
If there is no $h>i$ with $C_h\in\DDd$, put $l'=h-1$.
Otherwise let $l'\ge i$ be the smallest integer such that $C_{l'+1}\in\DDd$.
Then the path $(C_{l+1},C_{l+2},\dots,C_{l'})$
contains subpath $(A_1^1,\dots,A_n^1; A_1^2,\dots,A_n^2; \dots;
A_1^{N-5},\dots, A_n^{N-5})$; hence $l'-l\ge n(N-5)$.)

Denote by  $i_1<\dots<i_k$ those indices $i$ for which
$C_i\in\DDd$. Then the previous observation
easily implies that
$i_{j+n+1}-i_j>n(N-5)$ for every $j<k-n$.
Hence $h>i_{1+r(n+1)}-i_1>rn(N-5)$,
where $r=\lfloor(k-1)/(n+1)\rfloor$.
Thus $(k-1)/(n+1)-1 < r < m/(n(N-5))$
 and so
\begin{equation}\label{EQ:thmA:km}
 \frac{k_h}{h}
 < \frac{2}{N-5} + \frac{n+2}{h}\,.
\end{equation}
From this (\ref{EQ:theta_B}) follows.

\medskip

To construct our $Q$-Lipschitz map $F:X\to X$, we first
define a splitting $\tilde{\AAa}=\{\tilde{A}:\ A\in\AAa\}$
of $X$.
For every $t\in T$ we denote by $X_t$
the union of $\{t\}$ and of all components of $X\setminus \{t\}$ disjoint with $T$.
For any subset $M$ of $T$ put
$$
 X_M = \bigsqcup_{t\in M} X_t.
$$
For every arc $A\in \AAa$ we define a subcontinuum $\tilde{A}$ of $X$
in such a way that the following hold:
\begin{itemize}
	\item the system $\tilde{\AAa}=\{\tilde{A}:\ A\in\AAa\}$ is a splitting
	  of $X$ into continua;
	\item $\tilde{A}\supseteq A$, and
		$\tilde{A} = A$ for every $A=A_i^j$ with $j\ne N$
	 	and for every $A=A_n^{N-1,l}$;
	\item $X_{\interiorRel{T}{A}} \subseteq \tilde{A} \subseteq X_A$
	 for every $A\in\AAa$;
	\item for every $t\in T$ there is $A\in\AAa$ such that $X_t\subseteq \tilde{A}$;
	\item if $A\ne B$ then $\tilde{A}\cap \tilde{B}=A\cap B$;
	  hence $P_{\tilde\AAa}=\bigcup_{\tilde{A}\ne\tilde{B}} (\tilde{A}\cap \tilde{B})
	   \subseteq Q$.
\end{itemize}
Notice that if $A=[t,u]\in\AAa$ then $\tilde{A}$ is one of the following four sets:
$X_{(t,u)}      \sqcup  \{t\}  \sqcup  \{u\}$,
$X_{(t,u)}      \sqcup  X_t    \sqcup  \{u\}$,
$X_{(t,u)}      \sqcup  \{t\}  \sqcup  X_u$,
$X_A=X_{(t,u)}  \sqcup  X_t    \sqcup  X_u$.

Now we define a convex metric $d=d_N$ on $X$.
Let $d_T$ be a convex metric on $T$
such that
\begin{equation}\label{EQ:thmA:def_of_dT}
 g:(T,d_T)\to(T,d_T)
 \qquad\text{ is } Q \text{-linear with constant slope }\lambda=\lambda_N > 1,
\end{equation}
see Lemma~\ref{L:parryMapOnTree}.
Put $l_A=\lengthd{d_T}{A}$ for $A\in \AAa$; then
\begin{equation}\label{EQ:thmA:length_g(A)}
 \lengthd{d_T}{g(A)} = \sum\limits_{B:\ A\tofsmall{g} B} l_B
 =
 \lambda\cdot \lengthd{d_T}{A} = \lambda\cdot l_A
 \qquad
 \text{for every } A\in\AAa.
\end{equation}
For $A=[a_0,a_1]\in\AAa$ let $d_{\tilde{A}}=d_{\tilde{A},a_0,a_1}$
and $\CCc_{\tilde{A}}=\CCc_{\tilde{A},a_0,a_1}$ be a convex metric on $\tilde{A}$
and a dense system of subcontinua of $\tilde{A}$ obtained
from Proposition~\ref{P:Sp12a-ThmC} (with $A_1=A$);
by (e) we may assume that $d_T|_{A\times A} = c_A\cdot d_{\tilde{A}}|_{A\times A}$
for some $c_A>0$. If we put
$\tilde{l}_A=\lengthd{d_{\tilde{A}}}{A}$, then $c_A=l_A/\tilde{l}_A$ and
$c_A\in [l_A, 2l_A]$ by Proposition~\ref{P:Sp12a-ThmC}(a) and the choice of $q$.
Let $d$ be the unique convex metric on $X$ with
\begin{equation}\label{EQ:thmA:def_d}
 d|_{\tilde{A}\times\tilde{A}}  = c_A\cdot d_{\tilde{A}}
 \qquad\text{for every } A\in\AAa.
\end{equation}
Notice that $X$ has finite length:
\begin{equation}\label{EQ:thmA:XhasFiniteLength}
 \lengthd{d}{X}
 = \sum_{A\in\AAa} c_A 
 <
 \infty.
\end{equation}

Now we can define the map $F:X\to X$. First, for $A\in\BBb$, $A\not\subseteq R$
we have that $\tilde{A}=A$ and $\tilde{B}=B$ for every $B$ with $A\tofsmall{g} B$;
in this case we put $F_A=g|_A$. Second, if $A\in\BBb$ and $A\subseteq R$,
we put $F_A=g\circ r_A$, where $r_A:\tilde{A}\to A$
is a $(\gamma,1)$-LEL$^*$ retraction from Lemma~\ref{L:sp12a-arc}.
Third, if $A=[a_0,a_1]\in\DDd$, take the unique
$Q$-basic arc $B=[b_0,b_1]\in\AAa$ ($b_i=g(a_i)$ for $i=0,1$)
with $A\tofsmall{g} B$.
Let $f_A:(\tilde{A},d_{\tilde{A}}) \to (\tilde{B},d_{\tilde{B}})$
be a $(2,L_2)$-LEL map obtained from Proposition~\ref{P:Sp12a-ThmD}
and define
\begin{equation*}
 F_A:(\tilde{A},d)\to (\tilde{B},d),
 \qquad
 F_A(x)=f_A(x)
 \quad\text{for } x\in \tilde{A}
\end{equation*}
(here $d$ denotes the corresponding restriction of the metric).
By (\ref{EQ:thmA:length_g(A)}), $l_B=\lambda l_A$; thus
$\Lip(F_A)=(l_B/l_A)\cdot (\tilde{l}_A/\tilde{l}_B) \cdot \Lip(f_A)
\le 2\lambda L_2$
and $F_A$ is length-expanding
w.r.t.~the constant $(l_B/l_A)\cdot (\tilde{l}_A/\tilde{l}_B) \cdot 2
\ge \lambda$.
Thus,
\begin{equation}\label{EQ:thmA:F_A:LEL}
 F_A \text{ is } (\lambda, 2\lambda L_2)\text{-LEL}.
\end{equation}
Finally, define $F=F_N:X\to X$ by
\begin{equation}\label{EQ:thmA:def_of_F}
 F(x) = F_A(x)  \qquad\text{for } x\in \tilde{A}, \ A\in\AAa.
\end{equation}

For $A\in\AAa$ put $L_{\tilde{A}}=\lambda$ if $A\in\BBb$ and
$L_{\tilde{A}}=2\lambda L_2$ if $A\in\DDd$.
Using (\ref{EQ:thmA:F_A:LEL}) and (\ref{EQ:thmA:def_of_dT}) we
have that
\begin{equation}\label{EQ:thmA:F-is-Q-Lipschitz}
 F 
 \text{ is a } Q \text{-Lipschitz map w.r.t.~} \tilde{\AAa}, \,
 (L_{\tilde{A}})_{\tilde{A}\in \tilde{\AAa}}
 \quad\text{and}\quad M_F=M_g.
\end{equation}
Concerning the entropy of $F=F_N$ we use
Proposition~\ref{P:entropyOfPLipschitz}.
Since $M_{F_N}=M_{g_N}$ and $\lambda=\lambda_N$ is the maximal
eigenvalue of $M_{g_N}$, Proposition~\ref{P:entropyOfPLipschitz}
and (\ref{EQ:theta_B}) give
$$
 h(F_N)\le \log \lambda_N + \frac{4\log (2\lambda_N L_2)}{N-5}
 \,.
$$
Since $L_2$ does not depend on $N$,
(\ref{EQ:thmA:limsup_lambdaN}) gives
\begin{equation}\label{EQ:thmA:entropyFN}
 \limsup_{N\to\infty}  h(F_N)
 \le \limsup_{N\to\infty}  \log \lambda_N
 \le \max\{h(f),(1/n)\log 2 \} .
\end{equation}

\medskip

Now we show that $F=F_N$ is exactly Devaney chaotic provided $N$ is sufficiently large.
Put $\CCc=\bigcup_{A\in\AAa} \CCc_{\tilde{A}}$
and fix any $C\in\CCc$. For $k\in\NNN_0$ put $C_k=F^k(C)$ and take $A_k\in\AAa$
such that $C_k\cap \tilde{A}_k$ is non-degenerate.
Let $h\ge 0$ be such that
\begin{equation}\label{EQ:Ck-subsetneq-Ak}
 C_k \subseteq \tilde{A}_k
 \qquad\text{for every }k<h.
\end{equation}
Then, by (\ref{EQ:thmA:def_of_dT}),
(\ref{EQ:thmA:F_A:LEL}) and the choice of $r_A$ ($A\in\BBb,A\subseteq R$),
\begin{equation}\label{EQ:lengthCm}
 \length{C_h} \ge \gamma^{p_h} \cdot\lambda^h \cdot \length{C},
 \quad
 p_h = \card \{ k<h:\ A_k\subseteq R \text{ and } C_k\not\subseteq A_k\}.
\end{equation}
By the definition of $F_A$ for $A\subseteq R$ we have that
if $A_k\subseteq R$ and $C_k\not\subseteq A_k$ for some $0<k<h$,
then $A_{k-1}$ is equal to some $A_n^{N-1,l}$. Hence,
by (\ref{EQ:thmA:km}),
$$
 \frac{p_h}{h}
 \le \frac{k_h+1}{h}
 \le \frac{2}{N-5} + \frac{n+3}{h}
 \,.
$$
Fix a constant $\theta\in (1/\lambda,1)$ and assume that
$N$ is so large that
 $\gamma^{2/(N-5)} >\theta$.
Using (\ref{EQ:lengthCm}) we have
that
$\length{C_h}\ge (\lambda \theta)^h  \cdot \gamma^{n+3} \cdot \length{C}$.
Since $\length{C}>0$, $\lambda \theta>1$ and $X$ has finite length
by (\ref{EQ:thmA:XhasFiniteLength}),
we have that $h$'s satisfying (\ref{EQ:Ck-subsetneq-Ak})
are bounded from above. Hence
there is $h_0$ such that $C_{h_0}\not\subseteq \tilde{A}_{h_0}$
and thus
\begin{equation}\label{EQ:Cm-intersects-Q}
 C_{h_0} \quad\text{intersects } Q \text{ for some }
 h_0\in\NNN.
\end{equation}

Take a subcontinuum $C'$ of $C_{h_0}$ such that $C'\in\CCc$,
$C'\subseteq \tilde{A}_{h_0}$ and $C'\cap Q\ne\emptyset$
(this is possible due to (\ref{EQ:Cm-intersects-Q}) and the definition
of $\CCc$). By analogous arguments we can find $h'>0$ such that
$C''=F^{h'}(C')$ is not a subset of any $\tilde{A}$.
We may assume that $h'$ is the smallest integer with this property,
hence $F^{h'-1}(C')\subseteq \tilde{A}$ for some $A\in\AAa$.
By inspecting the transition graph of $g$ we see that either $A\subseteq R$
or $A\in\{A_n^{N-5}, A_n^{N-2}, A_{n-1}^{N-1,l}:\ l=1,\dots,m\}$.
Since $F^{h'-1}(C')$ intersects $Q$, the set $C''$
contains a $Q$-basic arc $B$ such that either $B\subseteq R$
or $B$ is a free arc in $X$. In the former case we have that
$F^k(B)\supseteq A_1^1$ for some $k$; so in both cases
\begin{equation*}
 \text{there is } h\in\NNN \text{ such that the set }
 C_h
 \text{ contains some } A\in\AAa
 \text{ with } \tilde{A}=A.
\end{equation*}
But now
$C_{h+k}\supseteq \bigcup_{g^k(A)\supseteq B} \tilde{B}$ for every $k>0$.
Since $g$ is exact we have that $C_{h+k}=X$ for every sufficiently large $k$.
Hence $F$ is exact.
Finally, since $X$ has a free arc which disconnects $X$, by
\cite[Theorem~1.1]{AKLS} the map $F$, being transitive,
has dense periodic points.

To summarize, we have proved that
$F_N$ is exactly Devaney chaotic for every sufficiently large $N$
and that
$\limsup_{N\to\infty}  h(F_N) \le \max\{h(f),(1/n)\log 2\}$.
Thus the proposition is proved.
\end{proof}

\medskip

Now we are ready to prove the results stated in the introduction.

\begin{proof}[Proof of Theorem~\ref{T:main}]
Let $X$ be a non-degenerate
dendrite such that no subtree of it contains all free arcs of $X$.
Fix $n=2^r$, $r\ge 1$ and $\eps>0$.
By Lemma~\ref{L:subtreeFreeArcs},
$X$ contains a subtree $T$ which is either an $n$-star or an $n$-comb such that
every terminal edge of $T$ has non-empty interior in $X$.
Since $T\in\TTt_0$, Lemma~\ref{L:YeIsPSLinear}
gives that
there is a $(P,S)$-linear Markov map $f:R\to R$
with $h(f)<(1/n)\log 2+\eps$,
where $P\subseteq T$ and $S=(s_0,s_1,\dots,s_n)$ for some $s_0\in T$
and the end points $s_1,\dots,s_n$ of $T$.
Now Proposition~\ref{P:thmA} gives that
$\IED(X)<(1/n)\log 2 + \eps$. Since $n=2^r$ and $\eps$ are arbitrary, we have $\IED(X)=0$.

The fact that the infima $\IT(X),\IED(X)$ are not attainable, follows
from \cite{DSS12}. Indeed, the space $X$, being a compact metric space which is not
a finite union of disjoint simple closed curve and which contains a free arc, admits no zero entropy transitive map.
\end{proof}

In the proof of Theorem~\ref{T:main} we have used Lemma~\ref{L:YeIsPSLinear},
which is a reformulation of Lemmas~4.6 and 4.7 from \cite{Ye}.
The use of Lemma~\ref{L:YeIsPSLinear} can be replaced by
Lemmas~\ref{L:specialMapsOnStars}, \ref{L:specialMapsOnCombs} which we prove in
Appendix~1. This makes the proof of Theorem~\ref{T:main} independent of
Ye's results.

\begin{proof}[Proof of Corollary~\ref{C:totallyRegular}]
Let $X$ be a non-degenerate completely regular continuum which is not a tree.
If $X$ is a dendrite then $\IED(X)=0$ by Theorem~\ref{T:main}, since
free arcs are dense in $X$ by Lemma~\ref{L:subtreeFreeArcs}.
Otherwise $X$ contains a simple closed
curve $S$. Since $X$ is completely regular, $S$ has non-empty interior and
thus $X$ contains a free arc; of course, this arc does not disconnect $X$.
So $\IED(X)=0$ by \cite{Sp12b}.
\end{proof}

\begin{proof}[Proof of Proposition~\ref{P:upperBoundFreeArc}]
Just use Proposition~\ref{P:thmA} and Lemma~\ref{L:YeIsPSLinear}.
\end{proof}

\section*{Appendix 1: Zero entropy $(P,S)$-linear Markov maps on stars and combs}

The purpose of this appendix is to give a simple construction
of $(P,S)$-linear Markov maps with zero entropy on stars and combs. This makes the proof
of Theorem~A independent on the construction of \cite{Ye}.

\begin{lemma}\label{L:specialMapsOnStars} Let $T$ be an $n$-star with end points $s_1,\dots,s_n$ ($n\ge 2$).
Then there is a map $f:T\to T$, a finite subset $P$ of $T$ and $s_0\in P$ such that
 $f$ is a $(P,S)$-linear Markov map with zero-entropy, where $S=(s_0,s_1,\dots,s_n)$.
\end{lemma}
\begin{proof}
 Let $b$ be the only branch point of $T$; so $T=\bigcup_{i=1}^n	[b,s_i]$. For every $i$ take a point
 $s_i'\in (b,s_i)$. Put $P=\{b;\ s_1',\dots,s_n';\ s_1,\dots,s_n\}$,
 $s_0=s_n'$ and $S=(s_0,s_1,\dots,s_n)$. Define the $P$-linear Markov map $f$ on $T$ by
 $f(b)=s_1'$, $f(s_n')=s_1$, $f(s_n)=s_1'$ and,
 for every $1\le i<n$, $f(s_i')=s_{i+1}'$ and $f(s_i)=s_{i+1}$
 (see Figure~\ref{Fig:4-star} for $n=4$).

\begin{figure}[ht!]
  \includegraphics{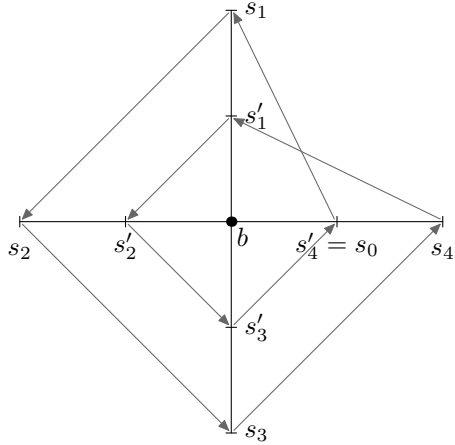}
  \caption{The map $f$ on a $4$-star}
  \label{Fig:4-star}
\end{figure}

 Obviously $f$ is $(P,S)$-linear.
 It is an easy exercise to show that $h(f)=0$.
 Indeed, one can use e.g.~the rome method with the rome
 consisting of two arcs: $A=[b,s_n']$, $B=[s_n',s_n]$.
 Since there is no path from $B$ to $A$ and the length of every simple
 path from $A$ to $A$ and from $B$ to $B$ is $n$,
 the characteristic polynomial of the transition matrix of $f$ is
 (see (\ref{EQ:romeMethod}))
 $$
  (-1)^{2n-2}\cdot x^{2n} \cdot \left( \frac{1}{x^n}-1 \right)^2 = (x^n-1)^2,
 $$
 the largest positive root of which is $\lambda=1$. Hence the entropy of $f$ is $h(f)=\log 1=0$.
\end{proof}

\begin{lemma}\label{L:specialMapsOnCombs}
Let $n=2^r$ ($r\ge 1$) and let $T$ be an $n$-comb.
Then there are a finite set $P$ and a tuple $S=(s_0,s_1,\dots,s_n)$,
where $s_0\in P$ and $\{s_1,\dots,s_n\}=\End(T)$,
such that $T$ admits a $(P,S)$-linear map $f:T\to T$ with zero entropy.
\end{lemma}
\begin{proof}
First we introduce some notation for distinguished points of $T$, which will form $P$.
Put $\Sigma=\{0,1\}$ and $\Sigma_k^l = \bigcup_{i=k}^l \Sigma^i$ for
$0\le k \le l \le r$, where $\Sigma^0$ is the singleton $\{\theta\}$
($\theta$ denotes the empty word);
that is, $\Sigma_k^l$ is the set of all words $\alpha=\alpha_0\dots\alpha_{i-1}$
over the alphabet $\Sigma$ of the length $\abs{\alpha}=i\in\{k,\dots,l\}$.
The concatenation $\alpha\beta$ and the powers $\alpha^k$
for words $\alpha,\beta$ and integers $k\ge 1$ are defined in a natural way.
For every $k$ let $<$ denote the lexicographical ordering on $\Sigma^k$.
Notice that $\Sigma^k$ with the addition from the left to the right is a group;
for $\alpha\in\Sigma^k$ and $n\in\ZZZ$ we define
$\alpha+n$ as usual.

Fix $r\ge 1$ and $n=2^r$.
Put $a_\theta=1/2$ and
$$
 a_\alpha=\frac{\alpha_0}{2} + \frac{\alpha_1}{2^2}
   + \dots + \frac{\alpha_{k-1}}{2^k} + \frac{1}{2^{k+1}}
 \qquad \text{for } \alpha=\alpha_0\dots\alpha_{k-1}\in\Sigma_1^{r-1}.
$$
For $\alpha\in\Sigma_0^{r-1}$ put $b_{\alpha 0} = a_\alpha -2^{-(r+2)}$ and
$b_{\alpha 1} = a_\alpha +2^{-(r+2)}$.
Notice that the following are true:
\begin{enumerate}
	\item[(i)]   $b_{\alpha 0}<a_\alpha < b_{\alpha 1}$ for $\alpha\in\Sigma_0^{r-1}$;
	\item[(ii)]  $b_\beta < b_\gamma$ for every $\beta < \gamma$ from $\Sigma_1^r$
	  of the same length;
	\item[(iii)] $b_{\alpha 0 \beta} < b_{\alpha 0} < b_{\alpha 1} < b_{\alpha 1 \beta}$
	  for every $\alpha\in\Sigma_0^{r-2}$ and $\beta\in\Sigma_1^{r-\abs{\alpha}-1}$.
\end{enumerate}
Identify the points $a_\alpha,b_\beta\in (0,1)$ with the points
$(a_\alpha,0), (b_\beta,0)$ of the Euclidean plane and define $c_\gamma=(b_\gamma,1)$
for $\gamma\in \Sigma^r$. Without loss of generality we may assume that
the $n$-comb $T$ is given by
$$
 T = A \cup \bigcup_{\gamma\in\Sigma^r} B_\gamma,
$$
where $A=[b_{0^r},b_{1^r}]$ and $B_\gamma=[b_\gamma,c_\gamma]$ ($\gamma\in\Sigma^r$)
are segments; see Figure~\ref{Fig:4-comb} for an illustration.
Put
$$
 P =
 \{a_\alpha:\ \alpha\in\Sigma_0^{r-1}\}
 \cup
 \{b_\beta:\ \beta\in\Sigma_1^{r}\}
 \cup
 \{c_\gamma:\ \gamma\in\Sigma^{r}\}.
$$
Obviously, $P\supseteq \End(T)\cup\Branchpoints(T)$
since $\End(T) = \{c_\gamma:\ \gamma\in\Sigma^r\}$ and
$\Branchpoints(T) = \{b_\gamma:\ \gamma\in\Sigma^r,\, \gamma\ne 0^r,1^r\}$.
The $P$-basic arcs are
\begin{itemize}
	\item $A_{\alpha 0}=[b_{\alpha 0},a_\alpha]$, $A_{\alpha 1}=[a_\alpha,b_{\alpha 1}]$
	  for $\alpha\in\Sigma_0^{r-1}$;
	\item $B_{\alpha 0}=[b_{\alpha 0 1^k}, b_{\alpha 0}]$,
	  $B_{\alpha 1}=[b_{\alpha 1}, b_{\alpha 1 0^k}]$
	  for $\alpha\in\Sigma_0^{r-2}$ and $k=r-\abs{\alpha}-1$;
	\item $C_\gamma = [b_\gamma,c_\gamma]$ for $\gamma\in\Sigma^r$.
\end{itemize}
For a $P$-basic arc $D$, we define the \emph{level} of it, denoted by $\lev{D}$,
as the length of its index;
e.g.~$\lev{B_{\alpha 0}}=\abs{\alpha 0}$ and $\lev{C_\gamma}=r$.
For $1\le k\le r$ denote by $\TTx_k$ the system of all $P$-basic arcs $D$ with $\lev{D}\ge k$;
put $\TTx_{r+1}=\emptyset$.

\begin{figure}[ht!]
  \includegraphics{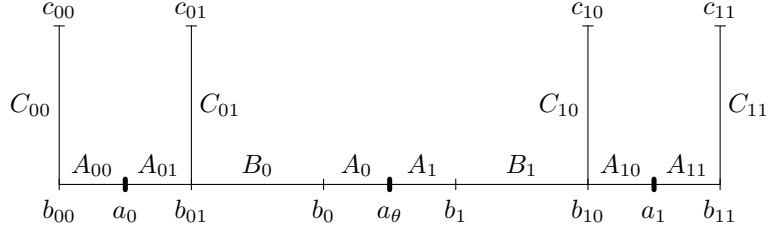}
  \caption{The $n$-comb $T$ for $n=2^2$}
  \label{Fig:4-comb}
\end{figure}

Define the map $f:T\to T$ by
\begin{enumerate}
	\item[(a)] $f(a_\alpha)= a_{\alpha+1}$ for $\alpha\in\Sigma_0^{r-1}$;
	\item[(b)] $f(b_{1^r})=c_{0^r}$, $f(b_{1^k})=b_{0^{k+1}}$ for $k<r$ and
	  $f(b_\beta)=b_{\beta+1}$ for $\beta\in\Sigma_1^r\setminus\{1^k:\ 1\le k\le r\}$;
	\item[(c)] $f(c_{1^r})=b_{0^r}$ and $f(c_\gamma)=c_{\gamma+1}$ for $\gamma\in\Sigma^r$,
	  $\gamma\ne 1^r$;
	\item[(d)] $f$ is $P$-linear.
\end{enumerate}
A simple analysis shows that the edges of the transition graph $G$ of $f$ are as follows
(here we write $D\to \TTx_k$ if $D$ $f$-covers zero ore more arcs from $\TTx_k$,
$1\le k \le r+1$):
\begin{itemize}
	\item $A_\alpha \to A_{\alpha+1}$ for $\alpha\not\in\{1^k: k\le r\}$;
	\item $A_{1^k} \to A_{0^k}, B_{0^{k}}, A_{0^{k+1}}, \TTx_{k+1}$ for $k<r$;
	\item $A_{1^r} \to A_{0^r}, C_{0^{r}}$;
	\item $B_\beta \to B_{\beta+1}, \TTx_{\abs{\beta}+1}$ for $\beta\not\in\{1^k: k< r\}$;
	\item $B_{1^k} \to A_{0^{k+1}}, \TTx_{k+1}$ for $k<r$;
	\item $C_\gamma \to C_{\gamma+1}$ for every $\gamma$.
\end{itemize}
Hence the loops of $G$ are
\begin{itemize}
	\item $A_{0^k}\to A_{0^k+1}\to \dots \to A_{1^k}\to A_{0^k}$, $k=1,\dots,r$;
	\item $C_{0^r}\to C_{0^r+1}\to \dots \to C_{1^r}\to C_{0^r}$.
\end{itemize}
Since every $P$-basic arc is contained in at most one loop of $G$, the entropy
of $f$ is zero.

Put $S=(s_0,s_1,\dots,s_{n})$, where $s_0=b_{1^r}$ and $s_i = c_{0^r + (i-1)}$
($i=1,\dots,n$).
To finish the proof we need to show that
$f$ is a $(P,S)$-linear Markov map. Indeed,
for $0\le i<n$ we have $f(s_i)=s_{i+1}$. Moreover, $s_{n}=c_{1^r}$ is an end point of $T$
and $[s_0,s_{n}]=C_{1^r}$ is a $P$-basic arc.
Finally, from the description of the transition graph of
$f$ we see that there is a path
from any $P$-basic arc $D$ to $C_{1^r}=[s_0,s_n]$.
\end{proof}

\section*{Appendix 2: Trees with large number of end points}

The lower and upper bounds for the infimum of entropies of transitive systems on a given tree $T$,
see \cite{ABLM, Ye, Bal01}, suggest that if the number of end points of $T$ is ``large''
then $T$ admits a transitive map with ``small'' entropy. In this appendix we show that this is in fact true,
even with transitivity replaced by exactness.
Moreover, we give an upper bound for the infimum depending only on the number of end points of the tree. This bound is in no sense optimal but it is very simple.
(In what follows, $\log$ means the natural logarithm.)

\begin{proposition}\label{P:limitOfI(T_n)} Let $T$ be a tree with $n$ end points. Then
\begin{equation}\label{EQ:explicitUpperBound}
 \frac{\log 2}{n}
 \le \IT(T)
 \le \IED(T)
 \le \frac{\log 2}{\sqrt{\log n}} \,.
\end{equation} 
Hence, if $(T_n)_n$ is a sequence of trees
with the number of end points going to infinity, then
$$
 \lim\limits_{n\to\infty} \IED(T_n) = 0.
$$
\end{proposition}

Before proving Proposition~\ref{P:limitOfI(T_n)} we prove the following
simple lemma.

\begin{lemma}\label{L:limitOfI(T_n)-4} Let $T$ be a tree with $n\ge 3$ end points
and such that every branch point of it has order at most $k$. Then $T$ has at least $n/k$ branch points.
\end{lemma}
\begin{proof} We prove the lemma by induction. If $n=3$ then $T$ is a $3$-star and the assertion is trivial ($k\ge 3$,
hence $n/k\le 1$). Assume that for some $n>3$ the assertion of the lemma holds for every tree (which is not an arc)
with the number of end points strictly smaller than $n$. Let $T$ be a tree with $n$ end points. Take a branch point
$b$ of $T$ (there is some since $T$ is not an arc).
Let $U_1,\dots, U_m$ ($m\le k$) be the components of $X\setminus\{b\}$; without loss of generality
we may assume that there is an integer $0\le m'\le m$ such that
$U_i$ is not an arc for every $i\le m'$ and $U_i$ is an arc for every $i> m'$.
For every $i\le m'$ put $n_i=\card{E(\closure{U_i})}$; then $n_i<n$ and so
the tree $\closure{U_i}$ has at least $n_i/k$ branch points by the induction hypothesis.
Easily we have that $n = \sum_{i=1}^{m'} (n_i-1) + (m-m')$, so $\sum_{i=1}^{m'} n_i \ge n-(m-2m') \ge n-k$. The
number of branch points of $T$ is at least $1+\sum_{i=1}^{m'} n_i/k\ge n/k$.
\end{proof}

\begin{proof}[Proof of Proposition~\ref{P:limitOfI(T_n)}] 
The lower bound from (\ref{EQ:explicitUpperBound}) was shown in \cite{ABLM}.
To show the upper bound take any tree $T$ with $n$ end points.
If $n=2$ we have $\IED(T)=(1/2)\log 2$,
so we may assume that $n\ge 3$.
Denote by $k\ge 1$ the only integer such that $\sqrt{\log n}-1 \le k <\sqrt{\log n}$.
By standard techniques one can show that
$2x\cdot\log x < x^2$ for $x\ge 1$
(indeed, $f(x)=x^2-2x\cdot \log x$ is strictly increasing on $[1,\infty)$
and $f(1)=1>0$),
hence $k^{2k} <  n$.

If $T$ has a branch point of order
at least $k+1$ then it contains a $(k+1)$-star, hence
$\IED(T)\le \log2 /(k+1)\le \log2 / \sqrt{\log n}$
by Proposition~\ref{P:upperBoundFreeArc} and the choice of $k$.
Otherwise every branch point of $T$ has order at most $k$.
By Lemma~\ref{L:limitOfI(T_n)-4} the tree $T$
has at least $n/k>k^{2k-1}$ branch points,
so by Lemma~\ref{L:limitOfI(T_n)-2} it contains a $(2k+2)$-comb.
Let $r\ge 1$ be the largest integer
such that $2^r\le 2k+2$; hence $2^{r+1}>2k+2$, that is, $2^r> k+1\ge \sqrt{\log n}$.
By Proposition~\ref{P:upperBoundFreeArc}
we have that $\IED(T)\le \log2/2^r < \log 2/\sqrt{\log n}$.
Hence the proof of (\ref{EQ:explicitUpperBound}) is finished.
\end{proof}


\end{document}